\documentclass{lmcs}
\usepackage[utf8]{inputenc}
\pdfoutput=1

\usepackage{lastpage}
\lmcsdoi{22}{4}{21}
\lmcsheading{}{\pageref{LastPage}}{}{}%
{Aug.~28,~2025}{Nov.~19,~2026}{}

\usepackage{tikz}
\usetikzlibrary{calc}
\usepackage{amssymb}
\usepackage{mathrsfs,hyperref}

\DeclareMathOperator{\Wr}{Wr}
\newcommand\qup[1]{\mathbb{Q}_\nearrow^{#1}}

\DeclareMathOperator{\Dom}{Dom}
\DeclareMathOperator{\Age}{Age}
\DeclareMathOperator{\Eq}{Eq}
\DeclareMathOperator{\Th}{Th}

\DeclareMathOperator{\Can}{Can}

\DeclareMathOperator{\nae}{NAE}

\newcommand{\bA}{\ensuremath{\mathfrak{A}}}
\newcommand{\bB}{\ensuremath{\mathfrak{B}}}
\newcommand{\bC}{\ensuremath{\mathfrak{C}}}
\newcommand{\bD}{\ensuremath{\mathfrak{D}}}
\newcommand{\bJ}{\ensuremath{\mathfrak{J}}}

\DeclareMathOperator{\Csp}{CSP}

\DeclareMathOperator{\Aut}{Aut}
\DeclareMathOperator{\End}{End}
\DeclareMathOperator{\Pol}{Pol}

\DeclareMathOperator{\Sym}{Sym}
\newcommand{\bF}{{\mathfrak F}}

\title{Structures preserved by primitive actions of \texorpdfstring{$S_\omega$}{S\_ω}}
\author[M.~Bodirsky]{Manuel Bodirsky\lmcsorcid{0000-0001-8228-3611}}[a]
\address{Institut f\"ur Algebra, Fakult\"at Mathematik, TU Dresden, Germany}
\email{manuel.bodirsky@tu-dresden.de}
\thanks{Both authors have been funded by the European Research Council (Project POCOCOP, ERC Synergy Grant 101071674). Views and opinions expressed are however those of the authors only and do not necessarily reflect those of the European Union or the European Research Council Executive Agency. Neither the European Union nor the granting authority can be held responsible for them. Manuel Bodirsky has also received funding from DFG (Project FinHom, Grant 467967530).}
\author[B.~Bodor]{Bertalan Bodor\lmcsorcid{0009-0003-6679-6355}}[b]
\address{Department of Set Theory, Logic, and Topology, HUN-REN Alfr\'{e}d R\'{e}nyi Institute of Mathematics, Budapest, Hungary}
\email{bodor@renyi.hu}

\keywords{first-order interpretation, primitive positive definition, primitive permutation group, primitive positive construction, first-order reduct, Ramsey expansion, finite boundedness, constraint satisfaction problem, polynomial-time tractability, NP-hardness}
\amsclass{03C10, 03C35, 03C40, 03C98, 68Q17}

\begin{document}

\begin{abstract}
We present a dichotomy for structures $\bA$
that are preserved by primitive actions  of $S_{\omega} = \Sym({\mathbb N})$:
such a structure primitively positively constructs all finite structures and the constraint satisfaction problem is NP-complete, or
the constraint satisfaction problem
for $\bA$ is in P.
To prove our result, we study the first-order reducts of the Johnson graph $\bJ(k)$, for $k \geq 2$, whose automorphism group $G$ equals the action of
$\Sym({\mathbb N})$
on the set $V$ of $k$-element subsets of $\mathbb N$.
We use the fact that $\bJ(k)$
has a finitely bounded homogeneous Ramsey expansion and that $G$ is a maximal closed subgroup of $\Sym(V)$.
\end{abstract}

\maketitle

\tableofcontents

\section{Introduction}
The 2017 solution of the finite-domain constraint satisfaction dichotomy conjecture of Feder and Vardi~\cite{FederVardi} has been a major success in a joint effort of several research areas: graph theory, finite model theory, and universal algebra.
If an arbitrary structure $\bB$ can primitively positively interpret all finite structures (up to homomorphic equivalence), then the constraint satisfaction problem (CSP) for $\bB$ is NP-hard. Otherwise, and if $\bB$ has a finite domain, then $\bB$ has a
\emph{cyclic polymorphism} (Barto and Kozik~\cite{Cyclic}), i.e., a homomorphism $c$ from $\bB^n$ to $\bB$, for some $n \geq 2$, such that for all $x_1,\dots,x_n \in B$
$$c(x_1,\dots,x_n) = c(x_2,\dots,x_n,x_1).$$
In this case, the constraint satisfaction problem for $\bB$ has been shown to be in P~\cite{Zhuk20} (the result has been announced independently by Bulatov~\cite{BulatovFVConjecture} and by Zhuk~\cite{ZhukFVConjecture}).

Many of the tools from universal algebra that have been used to obtain this result extend to large classes of structures $\bB$ over infinite domains, provided that the automorphism group of $\bB$ is large enough.
The corresponding class of CSPs is much larger;
see for instance~\cite{BG-Sandwich} for computational problems in graph theory and~\cite{BKR,BodirskyKnaeuerStarke} for large classes of problems from finite model theory that fall into this setting. It has been conjectured that if $\Aut(\bB)$ contains
the automorphism group of a relational structure $\bA$ which is
\begin{itemize}
\item \emph{homogeneous} (i.e., every isomorphism between finite substructures of $\bA$ extends to an automorphism of $\bA$), and
\item \emph{finitely bounded} (i.e., there is a finite set of finite structures $\mathcal F$ such that a finite structure $\bC$ embeds into $\bA$ if and only if no $\bF \in {\mathcal F}$ embeds into $\bC$),
\end{itemize}
then $\Csp(\bB)$ is in P or NP-complete (and similar characterisations of the NP-hard and the polynomial cases as in the finite have been conjectured as well~\cite{BPP-projective-homomorphisms,BKOPP-equations,wonderland,Topo}).
We refer to this conjecture as the \emph{infinite-domain CSP dichotomy conjecture}. It is wide open in general, but has been verified for numerous subclasses, usually with the help of a finitely bounded  homogeneous Ramsey expansion of $\bA$ (see~\cite{MottetPinskerSmoothConf,BodirskyBodorUIPJournal} for two recent results).

We approach the conjecture by studying classes that are tame from a model theory perspective. One of the most prominent model-theoretic restrictions is $\omega$-stability (see, e.g.,~\cite{CherlinHarringtonLachlan,LachlanIndiscernible,CherlinLachlan,HodgesHodkinsonLascarShelah,HrushovskiTotallyCategorical}; $\omega$-stability is a model-theoretic tameness condition on a complete first-order theory asserting that it has only countably many types over any countable parameter set, which implies strong structural control over its models). If we additionally impose the restriction that the structure $\bA$ is
$\omega$-stable, the conjecture is still far out of reach. In particular, it still contains
the class of all structures interpretable in $({\mathbb N};=)$ (also called structures \emph{`interpretable over equality'}).
This class is among the most basic classes in model theory -- it is the smallest class which contains infinite structures and is closed under first-order interpretations. Such structures are $\omega$-categorical\footnote{A structure is called $\omega$-categorical if all countable models of its first-order theory are isomorphic; this property is highly relevant for the study of the complexity of the CSP, because if $\bB$ is $\omega$-categorical, then the complexity of $\Csp(\bB)$ is captured by the polymorphisms of $\bB$, which is an essential ingredient to the universal-algebraic approach mentioned above.} and even reducts of finitely bounded homogeneous structures (i.e., obtained from finitely bounded homogeneous structures by dropping some relations).

The class of all structures interpretable over equality contains all structures that are $\omega$-categorical and \emph{monadically stable}, which in turn contains all finite structures. The infinite-domain CSP dichotomy conjecture has been verified for the class of all $\omega$-categorical monadically stable structures~\cite{Bodor24}. See Figure~\ref{fig:classes} for a chain of model-theoretically defined classes that are further steps towards the verification of the infinite-domain CSP dichotomy conjecture (the first inclusion from the top might be an equality~\cite{BPT-decidability-of-definability}). The third inclusion will appear in a forthcoming paper~\cite{bodirsky2025takingmodelcompletecores}. The references for the dichotomy results are~\cite{Bodor24,BodMot-Unary,Zhuk20} and the references for the equalities are~\cite{Braunfeld-Monadic-Stab,BodirskyBodor}.

\begin{figure}
\includegraphics[scale=.5]{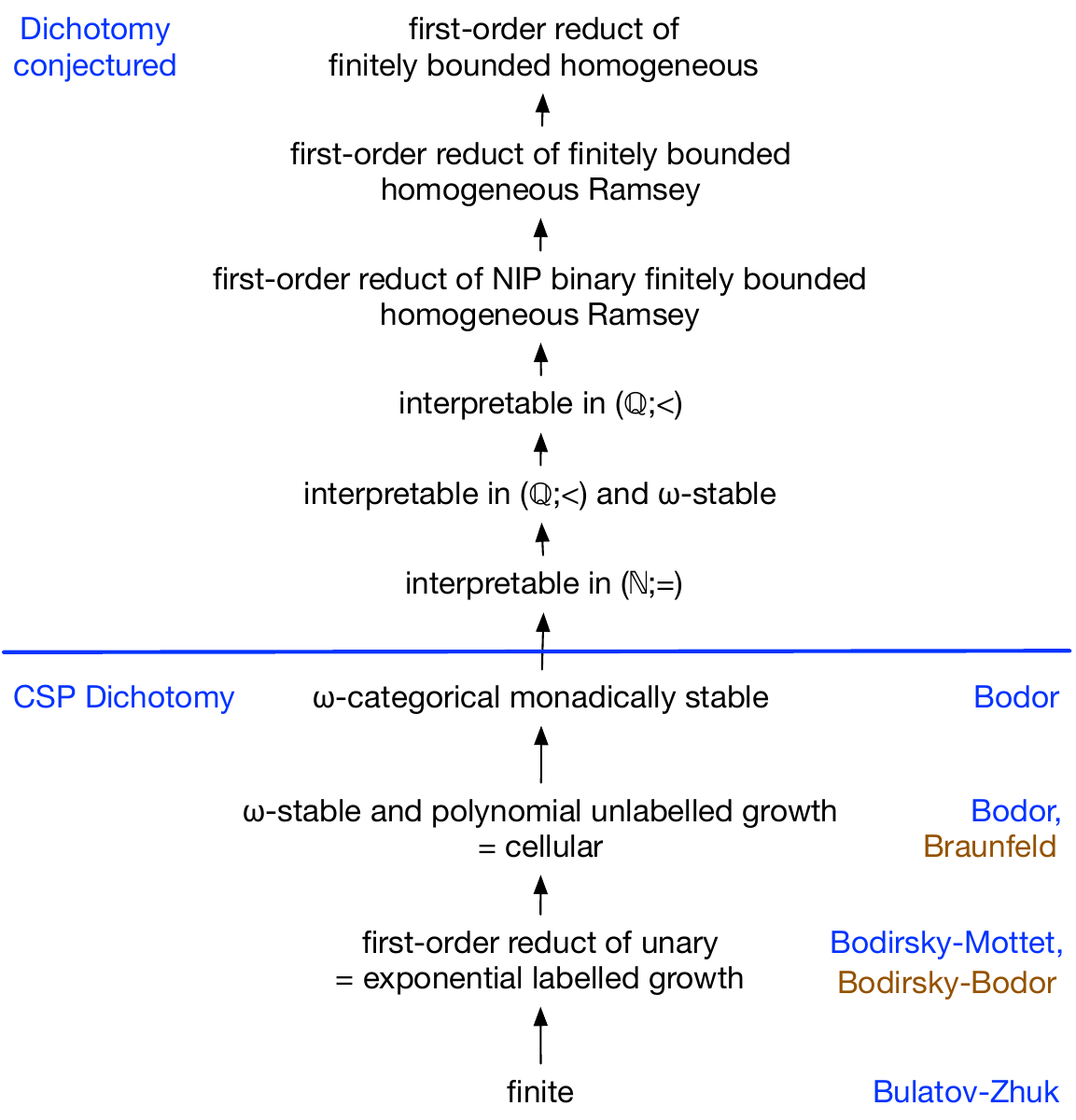}
\caption{The class of structures which are interpretable in $({\mathbb N};=)$ is one of the smallest model-theoretic classes where the infinite-domain CSP dichotomy conjecture has not yet been verified.
References about complexity of constraint satisfaction are marked in blue; references for model-theoretic equalities are marked in brown.}
\label{fig:classes}
\end{figure}

\noindent
Structures $\bB$ that are interpretable over equality can be fruitfully studied by investigating their automorphism group $\Aut(\bB)$. Typically, when studying permutation groups one first studies the primitive case, then uses the primitive case to study the transitive case, and finally uses the transitive case to describe the general case (see, e.g., Section~1.2 in~\cite{Oligo}, or Section~3 in~\cite{DixonMortimer} for more details on this approach). In this article, we focus on the primitive case, and verify the infinite-domain CSP dichotomy conjecture in the case that $\Aut(\bB)$ contains
a primitive action of $\Sym({\mathbb N})$.
Such structures $\bB$ can be equivalently described relationally, without reference to the automorphism group; for this, we need the following definitions.

A \emph{first-order reduct} of a structure $\bA$ is a structure with the same domain as $\bA$ whose relations are first-order definable in $\bA$.
If $\bA$ is $\omega$-categorical, then $\bB$ is a first-order reduct of $\bA$ if and only if $\Aut(\bA) \subseteq \Aut(\bB)$ (a well-known consequence of the theorem of Engeler, Svenonius, and Ryll-Nardzewski; see, e.g., Theorem~6.3.1 in~\cite{Hodges}).
The automorphism group of a structure $\bB$ contains a primitive action of $\Sym({\mathbb N})$ if and only if $\bB$
is (isomorphic to) a first-order reduct of a
particular undirected graph $\bJ(k)$, for some $k \in {\mathbb N}$, called the \emph{Johnson graph} (see Section~\ref{sect:J}), which is interpretable over $({\mathbb N};=)$.
The automorphism group of $\bJ(k)$ equals the action of $\Sym({\mathbb N})$
on the set $V$ of $k$-element subsets of $\mathbb N$ (see~\cite[Example~1.2]{Thomas96}).
Simon Thomas~\cite{Thomas96} showed that $\Sym(V)$ is the only proper closed supergroup of $\Aut(\bJ(k))$ for $k\geq 2$.

The automorphism group of a structure
does not carry sufficient information to determine the computational complexity of
 the CSP of the structure. If the structure is
 $\omega$-categorical, we only recover the structure from the automorphism group up to \emph{first-order interdefinability}.
However, the complexity of the CSP is preserved by
\emph{primitive positive bi-interpretability}.

Another important concept to study the computational complexity of CSPs is the concept of a \emph{model-complete core},
which is a structure $\bB$ such that every endomorphism of $\bB$ preserves all first-order formulas.
Every $\omega$-categorical structure is homomorphically equivalent to a model-complete core, which is unique up to isomorphism, and again $\omega$-categorical~\cite{Cores-journal,BodHilsMartin-Journal}.
Classifying the model-complete cores
of all structures whose automorphism group contains a given automorphism group can be quite challenging.\footnote{A missing classification of this type is the main obstacle for classifying the complexity of all CSPs for structures with the same automorphism group as Allen's Interval Algebra, mentioned e.g.\ in~\cite{Qualitative-Survey}.}

\subsection{Contributions}
We show that every countably infinite
structure whose automorphism group contains a primitive action of $\Sym({\mathbb N})$ is (isomorphic to) a first-order reduct of $\bJ(k)$, for some $k \in {\mathbb N}$ (Corollary~\ref{cor:Jk}), and all such structures are homomorphically equivalent to a model-complete core which is first-order bi-definable with $\bJ(\ell)$ for some $\ell \leq k$ (Theorem~\ref{thm:main}).
We also show that every
model-complete core structure $\bB$ which is first-order interdefinable with $\bJ(\ell)$, for $\ell \geq 2$,
primitively positively interprets all finite structures (Theorem~\ref{thm:J-stbb}).
From this, we obtain that the mentioned infinite-domain CSP dichotomy conjecture holds for the class of structures
that are preserved by a primitive action of $\Sym({\mathbb N})$ (Theorem~\ref{thm:main-complexity}).
In particular, for every reduct $\bB'$ of $\bB$ with finite relational signature, $\Csp(\bB')$ is in P or NP-complete.

We also obtain two strengthenings of these results, which are not needed for the complexity classification.
First, for $k \geq 2$, a first-order reduct of $\bJ(k)$ is a model-complete core that is first-order interdefinable with $\bJ(k)$ if and only if, for two distinct $i,j \in \{0,\dots,k-1\}$, the relation of intersecting in exactly $i$ elements and the relation of intersecting in exactly $j$ elements are primitively positively definable in it (Theorem~\ref{thm:ppdef}).
Second, every model-complete core structure which is first-order interdefinable with $\bJ(k)$, for $k \geq 2$, has only essentially unary polymorphisms (Theorem~\ref{thm:J-coll}).
In the terminology of Marimon and Pinsker~\cite{marimon2024minimal}, this means that $\Aut(\bJ(k))$ is \emph{collapsing}: the only clone whose unary operations are exactly the operations in $\overline{\Aut(\bJ(k))}$ is the clone generated by $\overline{\Aut(\bJ(k))}$.
Only very few examples of collapsing infinite oligomorphic permutation groups are known.
As a consequence, the model-complete core of a first-order reduct of $\bJ(k)$ has one element, or is first-order bi-definable with $({\mathbb N};=)$, or is primitively positively bi-definable with $({{\mathbb N} \choose \ell};S_0,S_1)$ for some $2 \leq \ell \leq k$, where $S_0$ is the disjointness relation and $S_1$ is the relation of intersecting in exactly one element (Corollary~\ref{main_pp}).

In our proof we make use of structural Ramsey theory.
A method using Ramsey theory and canonical functions was developed and used in earlier papers; see~\cite{BP-reductsRamsey} for an introduction. However, it can be observed that this method does not scale well, in the sense that for more complicated structures, the number of canonical functions and the number of cases that need to be considered in a combinatorial analysis grows tremendously. In our case, it is furthermore unclear how to perform such an analysis for $\bJ(k)$ for arbitrary $k$.

In this article we develop novel techniques to overcome this difficulty: instead of classifying all canonical functions, we collect enough partial information about them to conclude the proof of the statement above. On a high level, our technique is as follows.
First, we construct an expansion of $\bJ(k)$ by finitely many relations which is homogeneous, finitely bounded, and additionally has the \emph{Ramsey property}~\cite{Topo-Dynamics} (Section~\ref{sect:RamseyExp}).
As in~\cite{BP-reductsRamsey}, we use canonical functions of the Ramsey expansion $\bJ^<(k)$ to classify the model-complete cores of first-order reducts of $\bJ(k)$. The approach from~\cite{BP-reductsRamsey} cannot be adapted without further refinement, because already in the case of $\bJ^<(2)$ there are too many orbits of pairs and thus far too many behaviours of canonical functions to be studied exhaustively.
In contrast, our approach uses the concept of \emph{(canonical) range-rigid maps} from~\cite{MottetPinskerCores}.
Our key and novel technical contribution is the analysis of canonical functions and of range-rigid canonical functions based on their action on definable equivalence relations in $\bJ^<(k)$. In Section~\ref{sect:defequiv} we develop a powerful formalism to reason about such relations and the maps that preserve them.

\section{Preliminaries}
We write $[n]$ for $\{1,\dots,n\}$ and
${\mathbb N}$ for $\{0,1,2,\dots\}$.
If $A$ is a set, $n \in {\mathbb N}$, and $a \in A^n$, then we also write $(a_1,\dots,a_n)$ for $a$.
If $f \colon A \to B$ is a function,
then we may also apply $f$ \emph{setwise} and \emph{componentwise}:
\begin{itemize}
\item  for $S \subseteq A$, we write $f(S)$ for $\{f(s) \mid s \in S\}$ (set-wise),
\item for $s \in A^n$, we write $f(s)$ for
$(f(s_1),\dots,f(s_n))$ (componentwise).
\end{itemize}
We write ${A \choose n}$ for the set of all subsets of $A$ of cardinality $n$.

If $E$ is an equivalence relation on $A$, then we write $A/E$ for the set of equivalence classes of $E$.
If $a \in A$, then $a/E$ denotes the equivalence class of $a$ with respect to $E$.

\subsection{Concepts from group theory}
If $A$ is a set, we write $\Sym(A)$ for the set of all permutations of $A$, which forms a group with respect to composition.
If $G$ and $H$ are groups, then we write $H \leq G$ if $H$ is a subgroup of $G$.
Subgroups of $\Sym(A)$, for some set $A$, are called \emph{permutation groups (on $A$)}.
If $G$ is a permutation group on $B$
and $A \subseteq B$, then we write
\begin{itemize}
\item $G_{A}$ for the subgroup of $G$ which consists of all $g \in G$ such that $g(a) = a$ for all $a \in A$ (the \emph{point stabiliser}),
\item $G_{\{A\}}$ for the subgroup of $G$ which consists of all $g \in G$ such that $g(A) = A$  (the \emph{set stabiliser}).
\end{itemize}
An \emph{orbit} of a permutation group $G$ on $A$ is a set of the form
$\{g(a) \mid g \in G\}$, for some $a \in A$.

A permutation group $G$ on an at most countable set $A$ is called \emph{oligomorphic} if for every $n \in {\mathbb N}$, the permutation group on $A^n$ given by
$\{a \mapsto g(a) \mid g \in G\}$
has finitely many orbits of $G$ (i.e., we consider the componentwise action of $G$ on $A^n$).
Note that if $G$ is an oligomorphic permutation group on $B$, and $A \subseteq B$ is finite, then $G_{A}$ is oligomorphic as well.

\subsection{Concepts from model theory}
All our formulas are first-order formulas without parameters (unless explicitly stated otherwise); formulas without free variables are called \emph{sentences}.
If $\bA$ is a structure, then the \emph{theory} $\Th(\bA)$ of $\bA$ is the set of all sentences that hold in $\bA$.
A structure $\bA$ is called \emph{$\omega$-categorical} if all countable models of $\Th(\bA)$ are isomorphic. An at most countable structure $\bA$ is $\omega$-categorical if and only if $\Aut(\bA)$,  the automorphism group of $\bA$, is oligomorphic~\cite[Theorem~6.3.1]{Hodges}.

As we mentioned before, a relational structure $\bB$ is called \emph{homogeneous} (sometimes also \emph{ultrahomogeneous})
if every isomorphism between finite substructures of $\bB$ can be extended to an automorphism of $\bB$.
It is easy to see from the above that countable homogeneous structures with a finite relational signature are $\omega$-categorical.

An $\omega$-categorical structure $\bB$ is a \emph{model-complete core} if every endomorphism of $\bB$ preserves all first-order formulas; equivalently, every first-order formula is equivalent to an existential positive formula over $\bB$ (see, e.g.,~\cite{Book}).

If $\bA$ and $\bB$ are relational structures on the same domain $A=B$, then
$\bA$ is a \emph{reduct} of $\bB$, and $\bB$ is called an \emph{expansion} of $\bA$, if
$\bA$ can be obtained from $\bB$ by dropping relations. Clearly, we then have
$\Aut(\bB) \subseteq \Aut(\bA)$.
An expansion $\bB$ of $\bA$ is called a \emph{first-order expansion} if all relations of $\bB$ are first-order definable in $\bA$.
Clearly, we then have
$\Aut(\bB) = \Aut(\bA)$.
 A \emph{first-order reduct} of $\bB$ is a reduct of a first-order expansion of $\bB$.

If $\bB$ is $\omega$-categorical, then
$\bA$ is a first-order reduct of $\bB$ \emph{if and only if} $\Aut(\bB) \subseteq \Aut(\bA)$~(see, e.g., \cite[Theorem~6.3.1]{Hodges}).
Two structures $\bA$ and $\bB$ are called \emph{(first-order) interdefinable} if $\bA$ is a first-order reduct of $\bB$ and vice versa.
Note that $\omega$-categorical structures are interdefinable if and only if they have the same automorphisms.
Two structures $\bA$ and $\bB$ are called
\emph{(first-order) bi-definable} if $\bA$ is interdefinable with a structure that is isomorphic to $\bB$.
Note that $\omega$-categorical structures $\bA$ and $\bB$ are bi-definable if and only if their automorphism groups are isomorphic as permutation groups.

If $\bA$ and $\bB$ are structures, $d \in {\mathbb N}$, $S \subseteq A^d$, and $I \colon S \to B$ is a surjection,
then $I$ is called a \emph{first-order interpretation (of $\bB$ in $\bA$)} if for every atomic formula $\phi$, including the formula $x=y$ and $x=x$, the pre-image of the relation defined
by $\phi$ in $\bB$ is first-order definable in $\bA$.
Structures with a first-order interpretation in an $\omega$-categorical structure are $\omega$-categorical as well.
If $I$ is a first-order interpretation of $\bB$ in $\bA$,
and $J$ is a first-order interpretation of $\bA$ in $\bB$ such that $I \circ J$ is first-order definable in $\bB$ and $J \circ I$ is first-order definable in $\bA$, then we say that $\bA$ and $\bB$ are \emph{first-order bi-interpretable}.

If $\bA$ is a relational structure and $B \subseteq A$ then we write $\bA[B]$ for the substructure
of $\bA$ induced on $B$.

\subsection{Concepts from constraint satisfaction}
For a structure $\bB$ with a finite relational signature $\tau$,
we write $\Csp(\bB)$ for the class of all finite $\tau$-structures with a homomorphism to $\bB$; we view this as a computational decision problem, namely the problem to decide for a given finite $\tau$-structure $\bA$ whether it has a homomorphism to the fixed structure $\bB$.
A formula is called \emph{primitive positive (pp)} if it does not contain negation, universal quantification, and disjunction. Up to logical equivalence, these are precisely the formulas of the form
$\exists x_1,\dots,x_n (\psi_1 \wedge \dots \wedge \psi_m)$ for some atomic formulas $\psi_1,\dots,\psi_m$.
A formula is called \emph{existential positive} if it does not contain negation and universal quantification,
but in contrast to primitive positive formulas disjunction is allowed.
It is easy to see that every existential positive formula is equivalent to a disjunction of primitive positive formulas.

An interpretation is called \emph{primitive positive} (or \emph{existential positive}) if all the defining formulas of the interpretation are primitive positive (or existential positive, respectively); similarly, we generalise interdefinability, bi-definability, and bi-interpretability.
If $\bB$ is homomorphically equivalent to a structure which has a primitive positive interpretation in $\bA$, then $\bB$ is called \emph{primitively positively constructible} in $\bA$~\cite{wonderland}.
The relevance of primitive positive constructions for constraint satisfaction comes from the following fact.

\begin{thmC}[{\cite[Theorem~3.1.4]{Book}}]
\label{thm:ppc}
If $\bA$ and $\bB$ have finite relational signatures and $\bB$ has a primitive positive interpretation in $\bA$, then there is a polynomial-time reduction from $\Csp(\bB)$ to $\Csp(\bA)$.
\end{thmC}

\noindent
Note that Theorem~\ref{thm:ppc} is true for primitive positive constructions instead of primitive positive interpretations, since homomorphic equivalence preserves the CSP.
The structure $(\{0,1\};\{(0,0,1),(0,1,0),(1,0,0)\})$ has an NP-complete CSP;
hence, Theorem~\ref{thm:ppc} implies that
if this structure (equivalently, all finite structures; see~\cite{Book}) has a primitive positive construction in $\bB$, then $\Csp(\bB)$ is NP-hard.
The \emph{infinite-domain tractability conjecture} states that if $\bB$ is a reduct of a finitely bounded homogeneous structure and $\bB$ has a finite relational signature,
and this structure does not have a primitive positive construction in $\bB$, then $\Csp(\bB)$ is in P.

We now recall some universal-algebraic tools to study primitive positive definability
in $\omega$-categorical structures.
If $\bA$ is a structure, then $\End(\bA)$ denotes the set of all \emph{endomorphisms} of $\bA$, i.e., homomorphisms from $\bA$ to $\bA$.
Clearly, $\End(\bA)$ forms a monoid with respect to composition.
Endomorphisms of $\bA$ preserve all relations with an existential positive definition in $\bA$.
An injective endomorphism that also preserves the complements of all relations of $\bA$ is called an \emph{embedding}.

We consider the set $B^B$ of all functions from $B$ to $B$ equipped with the topology of pointwise convergence: i.e., we equip $B$ with the discrete topology and $B^B$ with the product topology.
It is well known that a submonoid $S \subseteq B^B$ is closed with respect to  this topology if and only if
 $S = \End(\bB)$ for some relational structure $\bB$.
 If $S \subseteq B^B$, we write $\overline S$ for the topological closure of $S$.
 In contrast, the smallest submonoid of $B^B = \End(B;=)$ that contains $S$ will be denoted by
 $\langle S \rangle$.

A \emph{polymorphism} of a structure $\bA$  is a homomorphism from $\bA^k$ to $\bA$, for some $k \in {\mathbb N}$;
it is well-known and easy to check that polymorphisms preserve all primitively positively definable relations of $\bA$. If $\bA$ is $\omega$-categorical, the converse holds as well~\cite{BodirskyNesetrilJLC}. Note that $\Pol(\bA)$ always forms a \emph{clone}, i.e., it contains all projection maps $\pi_i^k\colon (x_1,\dots,x_k)\mapsto x_i$, and it is closed under compositions. The topology of pointwise convergence on clones can be defined similarly as for unary operations.
It is well-known that a clone is closed with respect to this topology if and only if it is a polymorphism clone of some structure (see, e.g.,~\cite{Book}).
If $S$ is a set of operations,
and $f$ is an operation in the smallest clone that contains $S$ and is closed with respect to the topology of pointwise convergence,
then we say that $S$ \emph{locally generates $f$.}

Let $\bA$ be an $\omega$-categorical structure.
A function $f \colon A \to A$ is \emph{canonical with respect to $\bA$} if for every $k \in {\mathbb N}$ and every $a \in A^k$, the orbit of $f(a)$ only depends on the orbit of $a$
with respect to the componentwise action
of $\Aut(\bA)$ on $A^k$.
We write $\Can(\bA)$
for the set of all operations on $A$ which are canonical with respect to $\bA$; see~\cite{BP-canonical} for an introduction to the applications of canonical functions.
Lemma~\ref{lem:canon} below can be used to find canonical functions in \emph{Ramsey structures}. For this article, the precise definition of the Ramsey property will be irrelevant; all that matters are the facts that
\begin{itemize}
\item $({\mathbb Q};<)$ has the Ramsey property (which is essentially Ramsey's theorem~\cite{Ram:On-a-problem}), and
\item whether a homogeneous structure has the Ramsey property only depends on its automorphism group, viewed as a topological group~\cite{Topo-Dynamics}.
\end{itemize}
We use the Ramsey property exclusively via the following lemma.
\begin{lemC}[Canonisation lemma~\cite{BPT-decidability-of-definability}, also see Theorem~5 in~\cite{BP-canonical}]
\label{lem:canon}
Let $\bA$ be countable $\omega$-categorical and Ramsey.
Then for every $f \colon A \to A$ there exists a function
$$g \in \overline{ \{ \beta \circ f \circ \alpha \mid \alpha,\beta \in \Aut(\bA)
\} } $$
which is canonical with respect to $\bA$.
\end{lemC}

\section{Primitive Actions of \texorpdfstring{$\Sym({\mathbb N})$}{Sym(ℕ)}}
A permutation group $G$ on a set $B$ is called \emph{transitive} if for all $x,y \in B$ there exists $\alpha \in G$ such that $\alpha(x) = y$; it is called \emph{primitive} if it is transitive and
$\Delta_B  := \{(x,x) \mid x \in B\}$ and $B^2$ are the only
\emph{congruences} of $G$, i.e., the only
equivalence relations on $B$ that are preserved by $G$.
It is well known that if $B$ has at least two elements, then a permutation group $G$ is primitive if and only if $G$ is transitive and each point stabiliser $G_a := G_{\{a\}} = G_{\{\{a\}\}}$, for $a \in B$, is a maximal subgroup of $G$.

\subsection{The Strong Small Index Property}
A permutation group $G$ on $X$ has the \emph{strong small index property (SSIP)} if
for every subgroup $H$ of $G$ of countable index there exists a finite $A \subseteq X$ such that $G_A \leq H \leq G_{\{A\}}$ (pointwise and setwise stabilisers of $A$ in $G$, respectively).
The strong small index property is useful, because it allows one to recover information about a structure directly from its automorphism group, making it a key tool in reconstruction and automorphism-group classification results in model theory and infinite permutation groups.

\begin{thmC}[Dixon, Neumann, and Thomas~\cite{DixonNeumannThomas}]
\label{thm:SSIP}
$\Sym({\mathbb N})$ has the SSIP.
\end{thmC}

\noindent
One can use this theorem to prove the following.

\begin{prop}\label{prop:prim-sym}
Let $\xi \colon \Sym({\mathbb N}) \to \Sym({\mathbb N})$ be a homomorphism such that $\xi(\Sym({\mathbb N}))$ is a primitive
permutation group $G$. Then there exists
a $k \in {\mathbb N}$
such that $G$ is isomorphic (as a permutation group)
to the setwise action of
$\Sym({\mathbb N})$ on ${{\mathbb N} \choose k}$.
\end{prop}
\begin{proof}
As mentioned above, the primitivity of $G$ implies that for any $a \in {\mathbb N}$ the point stabiliser $G_a$ is a maximal subgroup of $G$, and hence $H := \xi^{-1}(G_a)$ is a maximal subgroup of $\Sym({\mathbb N})$,
and since $G$ acts transitively on ${\mathbb N}$, the subgroup $H$
has countably infinite index in $\Sym({\mathbb N})$.
The strong small index property of
$\Sym({\mathbb N})$ (Theorem~\ref{thm:SSIP}) implies that
there exists some finite $F \subseteq {\mathbb N}$ such that
$$
\Sym({\mathbb N})_F
\leq H
\leq \Sym({\mathbb N})_{\{F\}}.
$$
The first inclusion implies that
$F\neq\emptyset$: otherwise,
$\Sym({\mathbb N})_F=\Sym({\mathbb N}) = H$, contrary to the assumption that $H$ is a proper subgroup of $\Sym({\mathbb N})$.
As $F$ is finite and nonempty,
$\Sym({\mathbb N})_{\{F\}}$ is a proper subgroup of
$\Sym({\mathbb N})$.
By the maximality of
$H$, this means that $H$ equals $\Sym({\mathbb N})_{\{F\}}$.

Let $k := |F| \geq 1$.
Let $i \colon {{\mathbb N} \choose k} \to {\mathbb N}$ be the map that sends
the set
 $\alpha(F) \in {{\mathbb N} \choose k}$ to
$\xi(\alpha)(a) \in {\mathbb N}$ for
each $\alpha \in \Sym({\mathbb N})$.

Note that $i$ is well-defined,
because if $\alpha, \beta \in \Sym({\mathbb N})$ are such that $\alpha(F) = \beta(F)$, then
$\alpha^{-1} \beta \in \Sym({\mathbb N})_{\{F\}} = H$, and hence $\xi(\alpha^{-1} \beta) \in G_a$.
Thus, $\xi(\alpha^{-1} \beta)(a) = a$. Since $\xi$ is a homomorphism
this gives
$\xi(\beta)(a) = \xi(\alpha)(a)$, as required for $i$ to be well-defined.

Also note that $i$ is bijective: surjectivity follows from the transitivity of $G$.
To see that $i$ is injective, let
$\alpha,\beta\in\Sym({\mathbb N})$ satisfy
$i(\alpha(F))=i(\beta(F))$.
Then $\xi(\alpha)(a)=\xi(\beta)(a)$, so
$\xi(\alpha^{-1}\beta)\in G_a$.
Hence,
$\alpha^{-1}\beta\in H=\Sym({\mathbb N})_{\{F\}}$,
which implies $\alpha(F)=\beta(F)$.
Thus, $i$ is bijective.
Moreover,
$i$ is an isomorphism between
the image of the setwise action of $\Sym({\mathbb N})$ on ${{\mathbb N} \choose k}$ and $G$:
for $\alpha \in \Sym({\mathbb N})$ and $S \in {{\mathbb N} \choose k}$, let $\beta \in \Sym({\mathbb N})$ be such that $\beta(F)=S$.
Then we have
\begin{align*}
i(\alpha(S)) & = i(\alpha \beta(F)) = \xi(\alpha \beta) (a) = \xi(\alpha) \xi(\beta) (a) = \xi(\alpha) i(\beta(F)) = \xi(\alpha) i(S). \qedhere
\end{align*}
\end{proof}

\subsection{The Johnson Graph}
\label{sect:J}
Fix $k \in {\mathbb N}$.
The Johnson graph $\bJ(k)$ is the graph
with vertex set ${{\mathbb N} \choose k}$ and edge set\footnote{We view graphs here as symmetric digraphs.}
$$E := \big \{ (S,T) \mid S,T \in {{\mathbb N} \choose k} \text{ and }  |S \cap T| = k-1 \big \}.$$
For $k=0$, we obtain a graph with a single vertex $\emptyset$ and no edges.
For $k=1$, we obtain a countably infinite clique, and for $k=2$ the so-called \emph{line graph} of the countably infinite clique.
For every $k \geq 0$, the graph $\bJ(k)$ is first-order interpretable in the $\omega$-categorical structure $({\mathbb N};=)$ and hence
$\omega$-categorical.
We will later need the following.

\begin{thmC}[Simon Thomas~{\cite[Example~1.2]{Thomas96}}]\label{thm:fo-reducts}
Let $H$ be the set-wise action of $\Sym({\mathbb N})$ on ${{\mathbb N} \choose k}$ and
let $G$ be a closed permutation group that strictly contains $H$. Then $G = \Sym({{\mathbb N} \choose k})$.
\end{thmC}

\noindent
Note that Theorem~\ref{thm:fo-reducts} has the following immediate consequences.

\begin{cor}\label{cor:interdef}
Any two structures with domain ${\mathbb N} \choose k$ that are preserved by the set-wise action of $\Sym({\mathbb N})$, but that are not preserved by $\Sym({{\mathbb N} \choose k})$, are first-order interdefinable.
\end{cor}

\begin{cor}\label{cor:interdef2}
	Every automorphism of the Johnson graph $\bJ(k)$ is induced by the set-wise action of a permutation in $\Sym({\mathbb N})$ on ${\mathbb N} \choose k$.
\end{cor}

\noindent
Combined with Proposition~\ref{prop:prim-sym}, we obtain the following.

\begin{cor}\label{cor:Jk}
Let $\bB$ be a countably infinite structure that is preserved by a primitive action of $\Sym({\mathbb N})$. Then $\bB$ is isomorphic to a first-order reduct of $\bJ(k)$, for some $k \in {\mathbb N}$.
\end{cor}
\begin{proof}
By Proposition~\ref{prop:prim-sym}, for some $k\in{\mathbb N}$
we may identify the domain of $\bB$ with ${{\mathbb N}\choose k}$
so that $\bB$ is preserved by the setwise action of
$\Sym({\mathbb N})$.
By Corollary~\ref{cor:interdef2},
$\Aut(\bJ(k)) \subseteq \Aut(\bB)$, and hence
$\bB$ is a first-order reduct of $\bJ(k)$.
\end{proof}

\section{Primitive Positive Definitions over the Johnson Graph}
This section contains a number of elementary but highly useful primitive positive definitions of some relations over the Johnson graph in terms of others. These results will then be used throughout the text. They culminate in a proof that every model-complete core structure which is first-order interdefinable with $\bJ(k)$, for some $k \geq 2$, admits primitive positive interpretations of all finite structures (Theorem~\ref{thm:J-stbb}).
Fix $k \geq 1$. For $i \in \{0,1,\dots,k\}$, let
\begin{align*}
& S_{ \leq i} := \big \{ (u_1,u_2)
\mid u_1,u_2 \in {{\mathbb N} \choose k} \text{ and } |u_1 \cap u_2| \leq i \big \}, \\
& S_{ \geq i} := \big \{ (u_1,u_2)  \mid u_1,u_2 \in {{\mathbb N} \choose k} \text{ and } |u_1 \cap u_2| \geq i \big \}, \\
\text{and}\quad & S_i := S_{\leq i} \cap S_{\geq i} .
\end{align*}
Note that $S_k$ is the equality relation on ${\mathbb N} \choose k$ and
that $S_{k-1}$ equals $E$.

\begin{lem}\label{lem:down}
For every $i,j \in \{0,\dots,k\}$ with $i+j \geq k$, the relation  $S_{\geq (i+j-k)}$ is primitively positively  definable in
$({{\mathbb N} \choose k};S_{\geq i},S_{\geq j})$.
\end{lem}
\begin{proof}
Note that $$\exists z \big (S_{\geq i}(x,z) \wedge S_{\geq j}(z,y) \big)$$
defines the relation $S_{\geq (i+j-k)}$.
\end{proof}

\begin{lem}\label{lem:flip}
For every $i \in \{0,\dots,k-1\}$, the relation  $S_{\leq k-i}$ is primitively positively  definable in both
$({{\mathbb N} \choose k};S_0,S_{\geq i})$ and $({{\mathbb N} \choose k};S_0,S_i)$.
\end{lem}
\begin{proof}
The primitive positive formulas
$$ \exists z \big ( S_0(x,z) \wedge S_{\geq i}(z,y) \big ) \text{ and } \exists z \big ( S_0(x,z) \wedge S_{i}(z,y) \big )$$
both define $S_{\leq k- i}$.
\end{proof}

\noindent
The following lemma is an easy consequence of the preceding two lemmata.

\begin{lem}\label{lem:Si}
For every $i \in \{1,\dots,k-2\}$ the relations $S_i$, $S_{\leq i}$, and $S_{\geq i}$ are
primitively positively  definable in
$({{\mathbb N} \choose k};E,S_0)$.
\end{lem}
\begin{proof}
Similarly as in Lemma~\ref{lem:down}, the formula $\exists z \big (E(x,z) \wedge E(z,y) \big)$
defines the relation $S_{\geq k-2}$. If $k=2$ then we are done. Otherwise, the formula $\exists z \big (S_{\geq k-2}(x,z) \wedge E(z,y) \big )$
defines the relation
$S_{\geq k-3}$.
Inductively, for $i \in \{2,\dots,k-1\}$
the formula
$\exists z \big (S_{\geq k-i}(x,z) \wedge E(z,y)  \big)$
defines the relation $S_{\geq k-i-1}$.

	By applying Lemma~\ref{lem:flip} we obtain that the relations $S_{\leq i}$, for $i\in \{1,\dots,k-1\}$, are primitively positively definable from the relations $S_0,E,S_{\geq k-2},\dots,S_{\geq 1}$. Finally, $S_i(x,y)$ is defined by $S_{\leq i}(x,y)\wedge S_{\geq i}(x,y)$.
\end{proof}

\subsection{The Sunflower Lemma}
\label{sect:sunflower}
Several times, we need the following Helly-type statement; it was proved by proved by Deza~\cite{Deza}. We give a short proof of a slightly weaker version which suffices for the purposes of the present article, for the convenience of the reader. Our first usage will be in Lemma~\ref{lem:Rn}, which is itself an important stepping stone towards primitive positive interpretations of all finite structures in Theorem~\ref{thm:J-stbb}.

\begin{lem}[Sunflower Lemma]\label{lem:intersect}
Let $m \in \{0,1,\dots,k\}$ and
$n \geq (k-1) {k \choose m} + 2$.
If $u_1,\dots,u_n \in {{\mathbb N} \choose k}$ pairwise satisfy $S_m$, then
$|u_1 \cap \cdots \cap u_n| = m$.
\end{lem}
\begin{proof}
By the pigeonhole principle we can find
a subset $I \subseteq \{2,\dots,n\}$ of size $k$ and an $r\in {u_1\choose m}$ such that
$u_1 \cap u_j = r$ for all $j \in I$.
We claim that in fact $r \subseteq u_j$ for all $j \in \{2,\dots,n\}$ which implies the statement of the lemma.

Suppose for contradiction that this is not the case, i.e., there exists $j \in \{2,\dots,n\}$ such that
$u_1 \cap u_j \neq r$. Clearly, $j \notin I$.
Let $i \in I \cup \{1\}$.
If $u_j \cap u_i = r$ then $u_j \cap u_1 = r$ in contradiction to the choice of $j$.
Since $|u_j \cap u_i| = m = |r|$ it follows that
$(u_j \cap u_i) \setminus r \neq \emptyset$.
Note that for every $x \in {\mathbb N} \setminus r$ there is at most one $i \in I \cup \{1\}$ such that $x \in u_i$. This implies that
$|u_j| \geq |u_j \setminus r| \geq |I\cup \{1\}|= k+1$, a contradiction.
\end{proof}

\begin{defi}
For  $m \geq 2$ and $i \in \{0,\dots,k\}$, let $S^m_i$ be the relation that contains all tuples $(u_1,\dots,u_m) \in {{\mathbb N} \choose k}^{m}$ such that $|\bigcap_{j \in [m]}  u_j| = i$ and $|u_p \cap u_q| = i$ for all $\{p,q\} \in {[m] \choose 2}$.
\end{defi}

\noindent
Note that $S^2_i = S_i$.

\begin{lem}\label{lem:Rn}
For every $m \geq 2$ and $i \in \{0,\dots,k\}$ the structure
$({{\mathbb N} \choose k};S^m_i)$ is primitively positively interdefinable with $({{\mathbb N} \choose k};S_i)$.
\end{lem}
\begin{proof}
The relation $S_i$ can be defined
in $({{\mathbb N} \choose k};S^m_i)$ by $\exists u_3,\dots,u_m. S^m_i(u_1,\dots,u_m)$.
To define the relation $S^m_i$ in $({{\mathbb N} \choose k};S_i)$,
let $n := \max((k-1) {k \choose i} + 2,m)$.
Then Lemma~\ref{lem:intersect} implies that
the following formula defines $S^m_i$.
\begin{align*}
\exists u_{m+1},\dots,u_{n} & \bigwedge_{p,q \in \{1,\dots,n\}, p \neq q} S_i(u_p,u_q)
\qedhere
\end{align*}
\end{proof}

\begin{lem}\label{lem:T}
For every $i \in [k]$, the structure $({{\mathbb N} \choose k};S_i)$ is primitively positively interdefinable with
$({{\mathbb N} \choose k};T_i)$ where
$T_i \subseteq {{\mathbb N} \choose k}^4$ is the relation
$$T_i := \{(u_1,u_2,u_3,u_4) \mid (u_1,u_2), (u_3,u_4) \in S_i, u_1 \cap u_2 = u_3 \cap u_4\}.$$
\end{lem}
\begin{proof}
Clearly, $S_i(u_1,u_2)$ is equivalent to $T_i(u_1,u_2,u_1,u_2)$, so
$S_i$ can be defined primitively positively from $T_i$.
Conversely, the relation $T_i$ has the definition
\begin{align*}
& \exists p,q \big (S^4_i(u_1,u_2,p,q) \wedge S^4_i(p,q,u_3,u_4) \big).
\end{align*}
The statement now follows from Lemma~\ref{lem:Rn}.
\end{proof}

\subsection{Existential Positive Bi-interpretability with \texorpdfstring{$({\mathbb N};\neq)$}{(ℕ;≠)}}
We write
${\mathbb N}_{\neq}^k$ for the subset of all tuples in ${\mathbb N}^k$ with pairwise distinct entries.

\begin{lem}\label{lem:biint}
The two structures $({\mathbb N};\neq)$ and $({{\mathbb N} \choose k};S_0,S_1)$ are existentially positively bi-interpretable.
\end{lem}
\begin{proof}
 Let $I$ be the $k$-dimensional interpretation
of $({{\mathbb N} \choose k};S_0, S_1)$ in $({\mathbb N};\neq)$ whose domain is
${\mathbb N}^k_{\neq}$ and which maps
$(x_1,\dots,x_k) \in {\mathbb N}_{\neq}^k$ to $\{x_1,\dots,x_k\}$. Clearly, the relations
\begin{align*}
& I^{-1}(\textstyle{{\mathbb N} \choose k}) = {\mathbb N}^k_{\neq} \\
& I^{-1}(=_{{{\mathbb N} \choose k}}) = \big \{(x,y) \mid x,y \in {\mathbb N}^k_{\neq} \text{ and } \{x_1,\dots,x_{k}\} = \{y_1,\dots,y_{k}\} \big \} \\
& I^{-1}(S_0) = \big \{(x,y) \mid x,y \in {\mathbb N}^k_{\neq} \text{ and } |\{x_1,\dots,x_{k}\} \cap \{y_1,\dots,y_{k}\}| = 0 \big \} \\
& I^{-1}(S_1) = \big \{(x,y) \mid x,y \in {\mathbb N}^k_{\neq} \text{ and } |\{x_1,\dots,x_{k}\} \cap \{y_1,\dots,y_{k}\}| = 1 \big \}
\end{align*}
 are existentially positively definable in $({\mathbb N};\neq)$.

 Let $J$ be the 2-dimensional interpretation
 of $({\mathbb N};\neq)$ in $({{\mathbb N} \choose k};S_0,S_1)$ whose domain is
 $S_1$ and which maps $(U_1,U_2) \in S_1$ to the element $u \in {\mathbb N}$ such that $\{u\} = U_1 \cap U_2$. It follows that $J^{-1}({\mathbb N}) = S_1$
 and that
 the relation
 \begin{align*}
J^{-1}(=_{\mathbb N}) & = T_1
 \end{align*}
 is primitively positively definable in $({{\mathbb N} \choose k};S_1)$ by Lemma~\ref{lem:T}.
Moreover, we have
\begin{align*}
J^{-1}(\neq_{\mathbb N}) = \big \{((U,V),(A,B))  \mid S_1(U,V) & \wedge S_1(A,B) \wedge \\
& \exists W,C (S^3_1(U,V,W) \wedge S^3_1(A,B,C) \wedge S_0(W,C)) \big  \}.
\end{align*}
Indeed, if $J(U,V)\neq J(A,B)$, choose disjoint
$S,T \in {{\mathbb N}\choose{k-1}}$
such that
$$(S\cup T)\cap(U\cup V\cup A\cup B)=\emptyset,$$
and put
$$W:=S\cup(U\cap V),\qquad C:=T\cup(A\cap B).$$
Then $S^3_1(U,V,W)$ and $S^3_1(A,B,C)$ hold.
Moreover, $W\cap C=\emptyset$, since
$U\cap V$ and $A\cap B$ are distinct singletons.
Hence, $S_0(W,C)$ holds as well.

Conversely, if $(U,V) \in S_1$ and $(A,B) \in S_1$ are such that there are disjoint $W,C$
with $S^3_1(U,V,W)$ and $S^3_1(A,B,C)$,
then in particular $U \cap V \neq A \cap B$.

 For all $(x_1,\dots,x_k),(y_1,\dots,y_k) \in {\mathbb N}^k_{\neq}$ and $z \in {\mathbb N}$
 we have
$$J \big (I(x_1,\dots,x_k),I(y_1,\dots,y_k) \big ) = z$$
if and only if $\{z\} = \{x_1,\dots,x_k\} \cap \{y_1,\dots,y_k\}$, which is
existentially positively definable in $({\mathbb N};\neq)$.
Moreover, for all $(U_1,V_1),\dots,(U_k,V_k)  \in S_1$ and
$P \in {{\mathbb N} \choose k}$ we have
$$I \big (J(U_1,V_1),\dots,J(U_k,V_k) \big ) = P$$
if and only if there are $W_1,\dots,W_{k}$ such that
$(U_i,V_i,W_i,P) \in T_1$ for all $i \in [k]$
and $(W_i,W_j) \in S_0$ for all $\{i,j\} \in {[k] \choose 2}$, and hence is existentially positively definable in $({{\mathbb N} \choose k};S_0,S_1)$.
\end{proof}

\begin{cor}\label{cor:j-is-mccore}
$\End({{\mathbb N} \choose k};S_0,S_1)$ is topologically isomorphic to $\End({\mathbb N};\neq)$.
In particular, the structure $({{\mathbb N} \choose k};S_0,S_1)$ is a model-complete core, and $\End({{\mathbb N} \choose k};S_0,S_1)=\overline{\Aut(\bJ(k))}$.
\end{cor}
\begin{proof}
Two $\omega$-categorical structures without constant endomorphisms are existentially positively bi-interpretable if and only if they have endomorphism monoids that are topologically isomorphic~\cite[Theorem~3.12]{BodJunker}.
An $\omega$-categorical structure $\bA$ is a model-complete core if and only if $\Aut(\bA)$ (the invertible elements in $\End(\bA)$) is dense in $\End(\bA)$, which can be seen from the endomorphism monoid
of $\bA$ as a topological monoid. Since $({\mathbb N};\neq)$ is clearly a model-complete core, this shows that
$({{\mathbb N} \choose k};S_0,S_1)$ is a model-complete core as well. By the same argument as above this implies that \[\End(\textstyle {{\mathbb N} \choose k};S_0,S_1)=\overline{\Aut({{\mathbb N} \choose k};S_0,S_1)}.\] Note, however, that by Corollary~\ref{cor:interdef} we know that $({{\mathbb N} \choose k};S_0,S_1)$ and $\bJ(k)$ are first-order interdefinable, in particular they have the same automorphism group. This implies the last equality.
\end{proof}

\subsection{Hardness}
In this section we show the NP-hardness of
$\Csp({{\mathbb N} \choose k};S_0,S_1)$ for $k\geq 2$.
In fact, we show a more general result for certain first-order reducts $\bC$ of $\bJ(k)$:
we assume that
\begin{itemize}
\item $\bC$ is a model-complete core, and
\item $\bC$ is even first-order interdefinable with $\bJ(k)$.
\end{itemize}

\noindent
By Corollary~\ref{cor:interdef} we know that if the second item above does not hold then $\Aut(\bC)=\Sym({{\mathbb N} \choose k})$, i.e., $\bC$ is interdefinable with $({{\mathbb N} \choose k};=)$, and thus bi-definable with $\bJ(1)$. Later, we will also show how to get rid of the first assumption above by describing the model-complete cores of the first-order reducts of $\bJ(k)$ to a sufficient degree of detail (Section~\ref{sect:mc-cores}).

\begin{thm}\label{thm:J-stbb}
Let $\bC$ be any model-complete core structure which is first-order interdefinable with $\bJ(k)$, for some $k \geq 2$.
Then $\bC$ interprets primitively positively all finite structures, and
$\bC$ has a reduct $\bC'$ with finite relational signature such that $\Csp(\bC')$ is NP-hard.
\end{thm}
\begin{proof}
Since $\Aut(\bC) = \Aut(\bJ(k))$ is the set-wise action of
$\Sym({\mathbb N})$
on ${{\mathbb N} \choose k}$ (see Corollary~\ref{cor:interdef2})
each of the relations $S_0, S_1$,\dots,
$N := S_{k-2}$,
$E := S_{k-1}$,
consists of a single orbit of pairs. Thus, since $\bC$ is a model-complete core, all these relations have a primitive positive definition in $\bC$ (see, e.g.,~\cite[Theorem~4.5.1]{Book}).
Therefore, the relation
$$\{(x,y) \mid x \neq y \} = \{(x,y) \mid \exists z (S_0(x,z) \wedge S_1(z,y) ) \}$$
has a primitive positive definition in $\bC$ as well (here we use the assumption that $k \geq 2$). Next, observe that $O := E \cup N$ has the definition
$$x \neq y \wedge \exists z \big (E(x,z) \wedge E(y,z) \big )$$
and hence is primitively positively definable in $\bC$. By Lemma~\ref{lem:Rn} we also know that the relations $S_{k-2}^3$ and $S_{k-1}^3$ are pp-definable from $N=S_{k-2}$ and $E=S_{k-1}$, respectively, and thus they are also pp-definable in $\bC$.
We may therefore assume without loss of generality that $\bC$ contains
 the relations $E$, $N$, $O$, $\neq $, $S_{k-2}^3$, and $S_{k-1}^3$.

We present a
2-dimensional primitive positive interpretation $I$
of the structure $$\bD := \big (\{0,1\}; \nae)$$ where $\nae$ is the ternary \emph{not-all-equal} relation, i.e., $\nae=\{0,1\}^{3}\setminus \{(0,0,0),(1,1,1)\}$. It is well known that this implies the statement (see, e.g.,~\cite[Theorem~3.2.2]{Book}; note that $\Csp(\bD)$ is essentially positive NAE-3-SAT and NP-complete by Schaefer's theorem). The domain formula is $O(x,y)$.
 Define $$I(x,y) := \begin{cases} 1 & \text{if } E(x,y) \\
 0 & \text{otherwise}.
 \end{cases}$$
Let $\phi(x_1,x_2,x_3)$ be the following primitive positive formula (illustrated in Figure~\ref{fig:S3-gadget})
 \begin{align}
 x_2 \neq x_1 \wedge x_2 \neq x_3 \wedge \exists y_1,y_2 \big (
  \bigwedge_{i \in \{1,3\}, j \in \{1,2\}} E(x_i,y_j) & \wedge x_2\neq y_1  \wedge x_2\neq y_2 \wedge S_{k-1}^3(x_2,y_1,y_2) \nonumber \\
 & \wedge E(y_1,y_2) \wedge N(x_1,x_3)
\big ).
\label{eq:phi}
 \end{align}

\begin{figure}
\centering
\begin{tikzpicture}[
    scale=1.2,
    every node/.style={circle,draw,inner sep=1.5pt}
]

\node (x1) at (-1.5,1.8) {$x_1$};
\node (y1) at (-1.5,0) {$y_1$};

\node (x3) at ( 1.5,1.8) {$x_3$};
\node (y2) at ( 1.5,0) {$y_2$};

\node (x2) at (0,-2.0) {$x_2$};

\fill[gray!20]
    ($(y1)!0.5!(y2)!0.82!(x2)$)
    --
    ($(y1)!0.18!(x2)$)
    --
    ($(y2)!0.18!(x2)$)
    --
    cycle;

\node[draw=none] at (0,-0.8) {$S_{k-1}^3$};

\draw[thick] (x1) -- (y1);
\draw[thick] (x3) -- (y2);
\draw[thick] (x1) -- (y2);
\draw[thick] (x3) -- (y1);

\draw[thick] (y1) -- node[above,draw=none] {$E$} (y2);

\draw[dashed,thick]
    (x1) -- node[above,draw=none] {$N$} (x3);

\end{tikzpicture}

\caption{
A graphical representation of the primitive positive formula $\phi(x_1,x_2,x_3)$ in~\eqref{eq:phi}; the inequalities between $x_2$ and the other variables have been omitted.
Solid lines represent the relation $E$, and the dashed line the relation $N$.}

\label{fig:S3-gadget}
\end{figure}

{\bf Claim 1.} $\phi(x_1,x_2,x_3)$ holds if and only if there exist an $A\in {{\mathbb N}\choose k-2}$ and pairwise distinct elements $p,q,r,s,t\in {\mathbb N}\setminus A$ such that $x_1=A\cup \{p,q\}, x_3=A\cup \{r,s\}$, and $x_2=A\cup \{t,u\}$ where $u\in \{p,q,r,s\}$. In particular, $\phi(x_1,x_2,x_3)$ implies
\[(E(x_1,x_2)\wedge N(x_2,x_3))\vee (N(x_1,x_2)\wedge E(x_2,x_3)).\]

	For the ``if'' direction, let
    $$\{v_1,v_2\} := \begin{cases} \{p,q\} \text{ if } u\in \{r,s\} \\
    \{r,s\} \text{ if } u\in \{p,q\}.
    \end{cases}$$
    Pick
    $y_1 =A\cup \{u,v_1\}$ and $y_2=A\cup \{u,v_2\}$. This choice witnesses the fact that $\phi(x_1,x_2,x_3)$ holds.

	For the other direction let us assume that $y_1,y_2 \in {{\mathbb N}\choose k}$ are such that the quantifier-free part of the formula $\phi(x_1,x_2,x_3)$ holds.
The conjunct $N(x_1,x_3)$ implies that
there are pairwise distinct $p,q,r,s \in {\mathbb N}$ and $A \in {{\mathbb N} \choose k-2}$ such that
$x_1 = A \cup \{p,q\}$, $x_3 = A \cup \{r,s\}$.
The conjuncts
$E(x_1,y_j)$ and $E(x_3,y_j)$ imply that $y_j=A\cup \{a_j,b_j\}$ where $a_j\in \{p,q\}$ and $b_j\in \{r,s\}$ for $j=1,2$. Since $E(y_1,y_2)$ holds we know that either $a_1=a_2$ and $b_1\neq b_2$, or $b_1=b_2$ and $a_1\neq a_2$. We will assume that the former holds; the other case is analogous. Let $u:= a_1=a_2$. The conjunction $x_2 \neq x_1 \wedge x_2 \neq x_3 \wedge S_{k-1}^3(x_2,y_1,y_2)\wedge x_2\neq y_1\wedge x_2\neq y_2$ implies that $x_2=A\cup \{u,t\}$ for some $t\not\in A\cup \{p,q,b_1,b_2\}=A\cup \{p,q,r,s\}$. This finishes the proof of the claim.

Let $\psi(x_1,x_2,y_1,y_2)$ be the primitive positive formula $\phi(x_1,x_2,y_1)\wedge \phi(x_2,y_1,y_2)$. By Claim 1 it is clear that $\psi(x_1,x_2,y_1,y_2)$ implies that either $E(x_1,x_2)\wedge E(y_1,y_2)$ or $N(x_1,x_2)\wedge N(y_1,y_2)$ holds.

\medskip
{\bf Claim 2.}
Let $F \subseteq {\mathbb N}$ be finite,
let \((u_1,u_2)\in O\) satisfy \(u_2\cap F=\emptyset\), and let \(a\in\mathbb N\setminus F\). Then there exists \((v_1,v_2)\in O\) such that \(\psi(u_1,u_2,v_1,v_2)\) holds and
\[
(v_1\cup v_2)\cap F=\emptyset
\qquad \text{ and }
\qquad
v_2\cap(F\cup\{a\})=\emptyset.
\]

We distinguish two cases depending on whether $(u_1,u_2)$ is in $N$ or in $E$. Let us first assume that $(u_1,u_2)\in N$ and let $A\coloneqq u_1\cap u_2$. Let $b\in u_2\setminus (A\cup \{a\})$; such an element exists, because $|A\cup \{a\}|\leq k-1$. Let $B$ be a $(k-2)$-element subset of $A\cup \{b\}\setminus \{a\}$; such a subset exists since $|A\cup \{b\}|=k-1$. Let $c,d,e$ be pairwise distinct elements outside $u_1\cup u_2\cup F\cup \{a\}$. Then the choice $v_1=A\cup \{b,c\},v_2=B\cup \{d,e\}$ suffices:
Since \(A\cup\{b\}\subseteq u_2\), both output sets avoid \(F\). Moreover, \(B\), \(d\), and \(e\) avoid \(a\), so \(v_2\) avoids \(a\). The verification of \(\psi\) is by Claim 1.

	Now let us assume that $(u_1,u_2)\in E$, and let $A$ be a $(k-2)$-element subset of $(u_1\cap u_2)\setminus \{a\}$, and let $b,c,d$ be pairwise distinct elements outside $u_1\cup u_2\cup F\cup \{a\}$. Then the choice $v_1=A\cup \{b,c\}$ and $v_2=A\cup \{b,d\}$ suffices: both $v_1$ and $v_2$ even avoid $F\cup\{a\}$, and again $\psi$ can be verified using Claim 1. This concludes the proof of the claim.

\medskip

	For $n\geq 1$, let $\psi_n(x_1,x_2,x_{2n+1},x_{2n+2})$
be the following primitive positive formula.
\[\exists x_3,\dots,x_{2n} \bigwedge_{i \in \{1,\dots,n\}}\psi(x_{2i-1},x_{2i},x_{2i+1},x_{2i+2}).\]

By induction, for every $n\geq 1$, the formula
$\psi_n(x_1,x_2,y_1,y_2)$ implies
$$
(E(x_1,x_2)\wedge E(y_1,y_2))
\vee
(N(x_1,x_2)\wedge N(y_1,y_2)).
$$
We show that the converse holds for $n=2k+2$.
Suppose, therefore, that $(x_1,x_2)$ and $(y_1,y_2)$
belong to the same one of the relations $E$ and $N$.
Put
$$
D:=x_1\cup x_2\cup y_1\cup y_2,
$$
enumerate $x_2=\{a_1,\dots,a_k\}$, and, for $0\leq j\leq k$,
define
$$
F_j:=(D\setminus x_2)\cup\{a_1,\dots,a_j\}.
$$
In particular, $F_0=D\setminus x_2$ and $F_k=D$.

Starting with $(u_1^{(0)},u_2^{(0)}):=(x_1,x_2)$, we
recursively apply Claim 2 to obtain pairs
$(u_1^{(j)},u_2^{(j)})\in O$ for $1\leq j\leq k$ such that
$$
\psi(u_1^{(j-1)},u_2^{(j-1)},u_1^{(j)},u_2^{(j)})
\quad\text{and}\quad
u_2^{(j)}\cap F_j=\emptyset.
$$
Indeed, $u_2^{(0)}\cap F_0=\emptyset$.
At step $j$, the inductive assumption gives
$u_2^{(j-1)}\cap F_{j-1}=\emptyset$, and
$a_j\notin F_{j-1}$.
Thus, Claim 2 applies with $F:=F_{j-1}$ and $a:=a_j$,
and its conclusion gives the required avoidance of
$F_j=F_{j-1}\cup\{a_j\}$.

After these $k$ applications, we have
$u_2^{(k)}\cap D=\emptyset$.
Apply Claim 2 once more, with $F:=D$ and
any $a\in{\mathbb N}\setminus D$.
This makes both output sets disjoint
from $D$.
Consequently, there exist $z_1,z_2\in{{\mathbb N}\choose k}$
such that
$$
\psi_{k+1}(x_1,x_2,z_1,z_2)
\quad\text{and}\quad
(z_1\cup z_2)\cap D=\emptyset.
$$
Moreover, the implication proved above gives
$$
|x_1\cap x_2|=|z_1\cap z_2|=|y_1\cap y_2|.
$$
Since $z_1\cup z_2$ is disjoint from $D$, it follows that
$(x_1,x_2,z_1,z_2)$ and $(z_1,z_2,y_1,y_2)$ lie in the
same orbit of $\Aut(\bJ(k))$.
Hence, $\psi_{k+1}(z_1,z_2,y_1,y_2)$ also holds.
Concatenating the two chains yields
$\psi_{2k+2}(x_1,x_2,y_1,y_2)$.
Therefore, $\psi_{2k+2}$ is a primitive positive definition
of the relation $I^{-1}(=_{\{0,1\}})$ in $\bC$.

	Let $\sigma(x_1,x_2,x_3,x_4)$ be the primitive positive formula

$$\exists y_1,y_2\bigwedge_{j\in \{1,2,3,4\}} S_{k-2}^3(y_1,y_2,x_j).$$

	Then $\sigma(x_1,x_2,x_3,x_4)$ holds if and only if there exists some $A\in {\mathbb{N} \choose k-2}$ such that $A \subseteq x_i$ for every $i \in \{1,\dots,4\}$. In this case let $x_i':= x_i\setminus A$ for $i\in \{1,2,3,4\}$. Then each of $x_1',x_2',x_3',x_4'$ is a 2-element subset of $\mathbb{N}\setminus A$.
	For $R \in \{E,N\}^6$, let
	$\phi_{R}$ be the formula (see Figure~\ref{fig:sigma-square})
	\begin{figure}
\centering
\begin{tikzpicture}[
    scale=0.9,
    every node/.style={circle,draw,inner sep=1.2pt}
]

\node (x1) at (-1.4, 1.4) {$x_1$};
\node (x2) at ( 1.4, 1.4) {$x_2$};
\node (x3) at ( 1.4,-1.4) {$x_3$};
\node (x4) at (-1.4,-1.4) {$x_4$};

\draw (x1) -- node[above=2mm,draw=none] {$R_1$} (x2);
\draw (x2) -- node[right=2mm,draw=none] {$R_2$} (x3);
\draw (x4) -- node[left=2mm,draw=none] {$R_4$} (x1);
\draw (x4) -- node[below=2mm,draw=none] {$R_6$} (x3);

\draw (x1) -- node[above left=.5mm,draw=none,yshift=3.5mm] {$R_3$} (x3);
\draw (x2) -- node[below left=.5mm,draw=none,yshift=-3.5mm] {$R_5$} (x4);

\end{tikzpicture}

\caption{
A graphical representation of the primitive positive formula $\phi_R$ (ignoring the conjunct $\sigma(x_1,x_2,x_3,x_4)$).
}
\label{fig:sigma-square}
\end{figure}
$$\sigma(x_1,x_2,x_3,x_4)\wedge R_1(x_1,x_2)\wedge R_2(x_2,x_3)\wedge R_3(x_3,x_1) \wedge R_4(x_1,x_4) \wedge R_5(x_2,x_4)\wedge R_6(x_3,x_4).$$
By checking all the possible cases for $x_1',x_2',x_3',x_4'$ one can easily verify the following.

\begin{itemize}
\item There do not exist $x_1,x_2,x_3,x_4\in {\mathbb{N} \choose k}$ such that $\phi_{(N,N,N,E,E,E)}(x_1,x_2,x_3,x_4)$ holds.
\item For any choice of $(R_1,R_2,R_3)\in \{E,N\}^3\setminus \{(N,N,N)\}$ there exist $x_1,x_2,x_3,x_4\in {\mathbb{N} \choose k}$ such that
$\phi_{(R_1,R_2,R_3,E,E,E)}(x_1,x_2,x_3,x_4)$ holds.
\item For any choice of $(R_4,R_5,R_6)\in \{E,N\}^3\setminus \{(E,E,E)\}$ there exist $x_1,x_2,x_3,x_4\in {\mathbb{N} \choose k}$ such that
$\phi_{(N,N,N,R_4,R_5,R_6)}(x_1,x_2,x_3,x_4)$
holds.
\end{itemize}
	From these observations and from the fact that $\psi_{2k+2}$
    defines $I^{-1}(=_{\{0,1\}})$, it follows that the relation $I^{-1}(\{0,1\}^3\setminus \{(0,0,0)\})$ is definable by the primitive positive formula
$$
\exists y_1,y_2,y_3,z \bigwedge_{i\in \{1,2,3\}} \psi_{2k+2}(x_i^1,x_i^2,y_i,y_{i+1}) \wedge \sigma(y_1,y_2,y_3,z)\wedge \bigwedge_{i\in \{1,2,3\}} E(y_i,z),
$$
where the addition in the indices is understood modulo 3. (Here and in the following formula, the free variables are $x_1^1,x_1^2,x_2^1,x_2^2,x_3^1,x_3^2$, where the pair $(x_i^1,x_i^2)$ represents the $i$-th argument of the ternary relation.)
Similarly, $I^{-1}(\{0,1\}^3\setminus \{(1,1,1)\})$ is definable by
$$
\exists y_1,y_2,y_3,z \bigwedge_{i\in \{1,2,3\}} \psi_{2k+2}(x_i^1,x_i^2,y_i,z) \wedge \sigma(y_1,y_2,y_3,z)\wedge \bigwedge_{i\in \{1,2,3\}} N(y_i,y_{i+1}).
$$
Since $I^{-1}(\nae)$ is the intersection of these relations, it follows that it has a primitive positive definition in $\bC$ as well.
\end{proof}

\section{The Model-Complete Cores of First-order Reducts of \texorpdfstring{$\bJ(k)$}{J(k)}}
\label{sect:mc-cores}
In this section we will show that the model-complete core of a first-order reduct of $\bJ(k)$ is
first-order bi-definable with
$\bJ(\ell)$, for some $\ell \leq k$ (Theorem~\ref{thm:main}).
The key to our analysis are functions ${{\mathbb N} \choose k} \to {{\mathbb N} \choose k}$ which are canonical with respect to a certain Ramsey expansion $\bJ^<(k)$ of $\bJ(k)$, which we will introduce in Section~\ref{sect:RamseyExp}. We then study these functions by their actions on the first-order definable equivalence relations; to this end, we first prove several fundamental properties of these equivalence relations (Section~\ref{sect:defequiv}) and then describe how functions that are canonical with respect to $\bJ^<(k)$ act on these relations (Section~\ref{sect:action}).

\subsection{A Finitely Bounded Homogeneous Ramsey Expansion}
\label{sect:RamseyExp}

We define
 $$\qup{k} := \{(a_1,\dots,a_k)\in \mathbb{Q}^k: a_1< \dots < a_k\}.$$

\begin{defi}\label{def:ramsey}
Let $\bJ^<(k)$ be the
structure whose domain is $\qup{k}$ and whose relations are
$<_{ij}$ and $=_{ij}$, for $i,j \in \{1,\dots,k\}$, defined by
\begin{align*}
(a_1,\dots,a_k) <_{ij} (b_1,\dots,b_k) & \text{ if and only if }
a_i < b_j \\
(a_1,\dots,a_k) =_{ij} (b_1,\dots,b_k) & \text{ if and only if }
a_i = b_j.
\end{align*}
\end{defi}

\noindent
Note that the relations $=_{ij}$ are quantifier-free definable from $<_{ij}$. The mapping
$$(a_1,\dots,a_k) \mapsto \{a_1,\dots,a_k\}$$
defines a bijection from $\qup{k}$ to ${{\mathbb Q} \choose k}$.
Via this bijection we can identify the elements of $\qup{k}$ with the $k$-element subsets of ${\mathbb Q}$. Also via a bijection between $\mathbb{Q}$ and $\mathbb{N}$ we can identify the $k$-element subsets of $\mathbb{Q}$ with the $k$-element subsets of ${\mathbb N}$, and hence we may assume that
$\bJ^<(k)$ has the same domain as $\bJ(k)$. Then the relation  $S_i$, for $i \in \{0,\dots,k\}$, can be written as a Boolean combination of the relations $=_{pq}$ for $p,q \in \{1,\dots,k\}$.
Indeed, for $a,b \in \qup{k}$ we have $(a,b) \in S_{\geq i}$ if and only if
$$\bigvee_{\substack{1\leq p_1<\cdots<p_i\leq k\\
                  1\leq q_1<\cdots<q_i\leq k}}
\ \bigwedge_{s=1}^{i} a=_{p_s q_s}b.$$
We can express $S_i$ as $S_i = S_{\geq i} \setminus S_{\geq i+1}$ for $i<k$.  In particular, using the identification above,
we may view $\bJ(k)$ as a first-order reduct of $\bJ^<(k)$.

It is well known and easy to see that a relational structure is finitely bounded if and only if its age has a finite universal axiomatisation.

\begin{thm}\label{thm:ramsey}
$\bJ^<(k)$ is homogeneous, Ramsey, and finitely bounded.
\end{thm}
\begin{proof}
For homogeneous structures, being Ramsey is a property of the
automorphism group as a topological group; see~\cite{Topo-Dynamics}. Therefore, it suffices to show the homogeneity of $\bJ^<(k)$ and that
the coordinatewise action of $\Aut({\mathbb Q};<)$ on $\qup{k}$, which we denote by $\xi$, defines an isomorphism between $\Aut({\mathbb Q};<)$ and $\Aut(\bJ^<(k))$ as topological groups.
 The fact that $\xi$ is an injective homomorphism from $\Aut({\mathbb Q};<)$ to $\Aut(\bJ^<(k))$ is clear. We also know that $\xi$ is continuous because agreement on the coordinates of finitely
many tuples implies agreement on these tuples. Moreover, $\xi^{-1}$ is also continuous on its image. Indeed, if $D$ is a finite subset of ${\mathbb Q}$ then we can choose a finite subset $F$ of $\qup{k}$
such that every element of $D$ occurs as a coordinate of a tuple in $F$. If the automorphisms induced by $\alpha_1$ and $\alpha_2$ agree on $F$, then $\alpha_1$ and $\alpha_2$ agree on $D$.
Thus, it remains to show that $\xi$ is surjective and that $\bJ^<(k)$ is homogeneous. Both of these follow from the following claim.

{\bf Claim 1.} For every isomorphism $f$ between finite substructures of $\bJ^<(k)$
there exists an automorphism $\alpha$ of $({\mathbb Q};<)$
such that
$f$ is the restriction of the map $\xi(\alpha)$ to $\Dom(f)$.

Let $A := \bigcup \Dom(f)$.
Then $|A| < \omega$. For an element $r \in A$ we define
$\beta(r)$ as follows. If $r = a_i$ for some $a \in \Dom(f)$,
then we define $\beta(r) := f(a)_i$. This is well-defined, because if $a_i = b_j$, then $a =_{ij} b$, and hence $f(a)_i = f(b)_j$ since $f$ preserves $=_{ij}$. Similarly, if $r < s$ for $r,s \in A$, then $r = a_i < b_j = s$ for some $a,b \in \Dom(f)$. Then we have $a <_{ij} b$ and since $f$ preserves $<_{ij}$ we have
$f(a)_i < f(b)_j$, which by definition implies that $\beta(r) < \beta(s)$. We obtained that $\beta$ is an isomorphism between finite substructures of $({\mathbb Q};<)$. Thus, by the homogeneity
of $({\mathbb Q};<)$ there exists an extension $\alpha \in \Aut({\mathbb Q};<)$ of $\beta$, which has the required properties.

Since $\xi$ is a homomorphism, Claim 1 immediately shows the homogeneity of $\bJ^<(k)$. In order to show that $\xi$ is surjective let $f\in \Aut(\bJ^<(k))$. Let us choose a sequence $F_0\subset F_1\subset\dots$ of finite subsets of $\mathbb{Q}$ such that $\bigcup_{n\in \mathbb{N}}F_n=\mathbb{Q}$ and $|F_0|\geq k$. By Claim 1 we can find a sequence $(\alpha_i)_{i\in \mathbb{N}}$ of automorphisms of $(\mathbb{Q};<)$ such that $\xi(\alpha_i)$ and $f$ agree on $F_i^k\cap \qup{k}$. Moreover, if $j>i$ then $\alpha_j$ and $\alpha_i$ agree on $F_i$. This means that the sequence $(\alpha_i)_i$ converges, let us denote its limit by $\alpha\in \overline{\Aut(\mathbb{Q};<)}$.
 Then for all $i\in \mathbb{N}$ we have $\alpha|_{F_i}=(\alpha_i)|_{F_i}$ therefore $\xi(\alpha)|_{F_i^k\cap \qup{k}}=\xi(\alpha_i)|_{F_i^k\cap \qup{k}}=f|_{F_i^k\cap \qup{k}}$, that is $\xi(\alpha)=f$. Moreover, $\alpha$ must be surjective since $f$ is surjective, and thus every $s \in {\mathbb Q}$ is a coordinate of some $f(a)$. Therefore, in fact $\alpha\in \Aut(\mathbb{Q};<)$.

	We have shown that $\xi$ is an isomorphism between the topological groups $\Aut({\mathbb Q};<)$ and $\Aut(\bJ^<(k))$, and hence $\bJ^<(k)$ is Ramsey.

For finite boundedness, observe that
$\bJ^<(k)$ satisfies for all $r,s,t \in [k]$ the following sentences (all variables are universally quantified):
\begin{align}
& x <_{rs} x && \text{ if } r<s \label{eq:qup} \\
& (x =_{11} y \wedge \cdots \wedge x =_{kk} y) \Leftrightarrow x = y \label{eq:refl} \\
& x <_{rs} y \vee y <_{sr} x \vee x =_{rs} y \label{eq:total} \\
& \neg (x =_{rs} y) \vee \neg (x <_{rs} y) \label{eq:antisym1} \\
& \neg (x <_{rs} y) \vee \neg (y <_{sr} x)
\label{eq:antisym2} \\
& (x <_{rs} y \wedge y <_{st} z) \Rightarrow x <_{rt} z \label{eq:trans}  \\
& x =_{rs} y \Rightarrow y =_{sr} x \label{eq:EqX1} \\
& (x =_{rs} y \wedge y =_{st} z) \Rightarrow x =_{rt} z \label{eq:EqX2} \\
& (x <_{rs} y \wedge y =_{st} z) \Rightarrow x <_{rt} z \\
& (x =_{rs} y \wedge y <_{st} z) \Rightarrow x <_{rt} z.
\end{align}

\noindent
It suffices to show that if a finite structure $\bA$
satisfies these sentences, then it has an embedding $\alpha$ into $\bJ^<(k)$.
We define binary relations $\sim$ and $<$ on $A \times [k]$ by
$$(a,i) \sim (b,j) \;\Leftrightarrow\; a =_{ij} b,
\qquad
(a,i) < (b,j) \;\Leftrightarrow\; a <_{ij} b.$$

{\bf Claim 2.}
$\sim$ is an equivalence relation, and $<$ induces
a linear order on $(A \times [k])/{\sim}$.

Reflexivity of $\sim$ follows from~\eqref{eq:refl}, symmetry from~\eqref{eq:EqX1},
and transitivity from~\eqref{eq:EqX2}.
The two final axioms involving both $=_{ij}$ and $<_{ij}$
show that replacing either argument of $<$ by a $\sim$-equivalent
element does not change whether the relation holds.
Thus, $<$ induces a well-defined relation on
$(A \times [k])/{\sim}$.
Transitivity follows from~\eqref{eq:trans}, and any two distinct
classes are comparable by~\eqref{eq:total}.
No class is less than itself by~\eqref{eq:antisym1}.

Write $[a,i]$ for the $\sim$-class of $(a,i)$.
Pick any embedding
$$e \colon ((A \times [k])/{\sim},<) \to ({\mathbb Q};<),$$
and define $\alpha \colon A \to {\mathbb Q}^k$ by
$$\alpha(a) := (e([a,1]),\dots,e([a,k])).$$

{\bf Claim 3.} $\alpha$ is an embedding of $\bA$ into $\bJ^<(k)$.

Note that~\eqref{eq:qup} implies that the image of $\alpha$
lies in $\qup{k}$. Injectivity follows from~\eqref{eq:refl}:
if $\alpha(a)=\alpha(b)$, then $[a,r]=[b,r]$ for all $r\in[k]$,
so $a =_{rr} b$ for all $r\in[k]$, and hence $a=b$.
Finally, for all $a,b\in A$ and $i,j\in[k]$, we have
\begin{align*}
\alpha(a) <_{ij} \alpha(b)
&\;\Leftrightarrow\; e([a,i]) < e([b,j])
\;\Leftrightarrow\; a <_{ij} b, \\
\alpha(a) =_{ij} \alpha(b)
&\;\Leftrightarrow\; e([a,i]) = e([b,j])
\;\Leftrightarrow\; a =_{ij} b.
\end{align*}
Thus, $\alpha$ preserves and reflects all relations.
\end{proof}

\begin{rem}\label{rem:inNP}
It follows that the CSP for all finite-signature first-order reducts $\bB$ of $\bJ(k)$ is in NP:
it suffices to show that $\bB$ is the reduct of a finitely bounded structure $\bC$ (see, e.g.,~\cite[Proposition~2.3.15]{Book}).
Using Theorem~\ref{thm:ramsey}, it is easy to see that the expansion of $\bB$ by the relations of $\bJ^<(k)$ is such an expansion.
\end{rem}

\subsection{Definable equivalence relations}
\label{sect:defequiv}
For a set $A$  and a binary relation $R \subseteq {A}^2$ we write
$\Eq(R)$ for the smallest equivalence relation that contains $R$.
We write $R^{\smile}$ for the relation $\{(y,x) \mid (x,y) \in R\}$.
Note that
\begin{align}
    \Eq(R) = \bigcup_{i \in \mathbb N} (\Delta_{A} \cup R \cup R^\smile)^i.
    \label{eq:Eq}
\end{align}
For a structure $\bA$ we write ${\mathcal E}(\bA)$ for the set of all equivalence relations that are first-order definable in $\bA$.
The set ${\mathcal E}(\bA)$ contains the inclusion-wise maximal element $A^2$
and the inclusion-wise minimal element $\Delta_{A}$.

For arbitrary equivalence relations $E,F$ on the same set,
we write
$E\vee F:=\Eq(E\cup F)$.
It is easy to see that if $\bA$ is $\omega$-categorical, then
$\Eq(E \cup F)$ is first-order definable in $\bA$ (since there are only finitely many inequivalent binary relations with a first-order definition in $\bA$, this follows from~\eqref{eq:Eq}), and hence an element of ${\mathcal E}(\bA)$ as well.

Next, we describe ${\mathcal E}(\bJ^<(k))$, the set of equivalence relations on $\qup{k}$ that are first-order definable in $\bJ^<(k)$.
Clearly, for every $I \subseteq [k]$
the relation $$E_I := \{(a,b) \in (\qup{k})^2 \mid a_i = b_i \text{ for all } i \in I\}$$
is a first-order definable equivalence relation on $\bJ^<(k)$.

\begin{rem}
Note that $E_{\emptyset} = (\qup{k})^2$.
Also note that if $I \subseteq J \subseteq [k]$, then $E_J \subseteq E_I$.
\end{rem}

\noindent
Let $a,b \in \qup{k}$. We write $(a,b)_=$
for the set $\{i \in [k] \mid a_i = b_i\}$.
Let $O$ be an orbit of pairs in $\bJ^<(k)$.
Define $$O_= := \{i \in [k] \mid a_i = b_i \text{ for some (equivalently, for every) } (a,b) \in O \}.$$

\begin{lem}\label{lem:eq}
Let $O$ be an orbit of pairs of $\Aut(\bJ^<(k))$. Then
$\Eq(O) = E_{O_=}$.
\end{lem}
\begin{proof}
For the containment $\Eq(O) \subseteq E_{O_=}$, it suffices to observe that $E_{O_=}$ is an equivalence relation which contains $O$.

To show the converse inclusion,
we first show that
$E_{[k] \setminus \{j\}} \subseteq \Eq(O)$ for every $j \in [k] \setminus {O_=}$.
 Let $(a,b) \in E_{[k] \setminus \{j\}}$. If $a = b$, then $(a,b) \in \Eq(O)$, so suppose that $a \neq b$.
Pick $(v,w)\in O$ and an automorphism
$\gamma\in\Aut({\mathbb Q};<)$ with $\gamma(v)=w$. By definition we know that $v_j\neq w_j$. If $v_j<w_j$ then the sequence $v_j,\gamma(v_j),\gamma^2(v_j),\dots$ is strictly increasing,
otherwise it is strictly decreasing. In either case we can find some $n$ such that $\gamma^n(v_j)\not\in \{v_1,\dots,v_k\}$. Let $u\coloneqq \gamma^n(v)$. Note that since $O$ is preserved by $\Aut(\bJ^<(k))$ we have
$(\gamma^m(v),\gamma^{m+1}(v))\in O$
for every $m\in\{0,\dots,n-1\}$,
and therefore $(u,v)\in\Eq(O)$ since $\Eq(O)$ is a transitive.
Choose $u_j'\in{\mathbb Q}$ with $u_j'>u_j$ such that
$(u_j,u_j']$ contains no coordinate of $u$ or $v$, and put
$$u':=(u_1,\dots,u_{j-1},u_j',u_{j+1},\dots,u_k).$$
By the homogeneity of $({\mathbb Q};<)$, there exists an
automorphism $\alpha$ which maps $u_j$ to $u_j'$ and fixes
all other coordinates of $u$ and all coordinates of $v$.
Thus, $\alpha(u)=u'$ and $\alpha(v)=v$.
Since $\Eq(O)$ is preserved by every automorphism of
$\bJ^<(k)$, we obtain $(u',v)\in\Eq(O)$, and hence
$(u,u')\in\Eq(O)$.
Also note that $(a,b)$ is in the same orbit as $(u,u')$ or in the same orbit as $(u',u)$;
in both cases, we obtain that $(a,b) \in \Eq(O)$.

To finally show $E_{O_=}\subseteq\Eq(O)$, let
$(a,b)\in E_{O_=}$ and define
$$c_i:=\max(a_i,b_i)\qquad\text{for every }i\in[k].$$
Then $c:=(c_1,\dots,c_k)\in\qup{k}$, and
$c_i=a_i=b_i$ for every $i\in O_=$.
For $i\in\{0,\dots,k\}$, put
\begin{align*}
a^i&:=(a_1,\dots,a_i,c_{i+1},\dots,c_k),\\
b^i&:=(b_1,\dots,b_i,c_{i+1},\dots,c_k).
\end{align*}
These tuples belong to $\qup{k}$: at the junction between
the two parts, we have $a_i,b_i\leq c_i<c_{i+1}$
whenever $1\leq i<k$.
Moreover, $a^0=b^0=c$, $a^k=a$, and $b^k=b$.
For every $j\in[k]\setminus O_=$, both
$(a^{j-1},a^j)$ and $(b^{j-1},b^j)$ belong to
$E_{[k]\setminus\{j\}}\subseteq\Eq(O)$.
For $j\in O_=$, we have $a^{j-1}=a^j$ and
$b^{j-1}=b^j$, so these pairs again belong to $\Eq(O)$.
By transitivity, $(a,c),(b,c)\in\Eq(O)$, and hence
$(a,b)\in\Eq(O)$.
\end{proof}

\begin{lem}\label{lem:join}
For $I,J \subseteq [k]$ we have $E_I \vee E_J = E_{I \cap J}$.
\end{lem}
\begin{proof}
Clearly, $E_{I \cap J}$ contains $E_I$
and it contains $E_J$, so it also contains
$E_I \vee E_J$. For the other inclusion let us consider the $k$-tuples $u=(1,2,\dots,k),v,w$ such that
\begin{itemize}
\item $v_i=w_i=i$ if $i\in I\cap J$,
\item $v_i=w_i=i+1/3$ if $i\in J\setminus I$,
\item $v_i=i, w_i=i+1/3$ if $i\in I\setminus J$,
\item $v_i=i+1/3, w_i=i+2/3$ if $i\in [k]\setminus (I\cup J)$.
\end{itemize}
	Then $u,v,w\in \qup{k}, (u,v)\in E_I, (v,w)\in E_J$, and thus $(u,w)\in E_I\vee E_J$. Let $O$ be the orbit of $(u,w)$. Then by definition ${O_=}=\{i\in [k]: u_i=w_i\}=I\cap J$, and thus by Lemma~\ref{lem:eq} we have $\Eq(O)=E_{O_=}=E_{I\cap J}$. Since $(u,w)\in O\cap (E_I\vee E_J)$, it follows that in fact $O\subseteq E_I\vee E_J$, and since $E_I\vee E_J$ is an equivalence relation we obtain $E_I\vee E_J\supseteq \Eq(O)=E_{I\cap J}$.
\end{proof}

\begin{cor}
${\mathcal E}(\bJ^<(k)) = \{ E_I \mid I \subseteq [k] \}$.
\end{cor}

\begin{proof}
Let $E \in {\mathcal E}(\bJ^<(k))$.
Let
$O_1,\dots,O_{\ell}$ be an enumeration of the orbits of pairs contained in $E$.
Then
\begin{align*}
E = \bigvee_{i = 1}^{\ell} \Eq(O_i) & =
\bigvee_{i = 1}^{\ell} E_{(O_i)_=} && \text{(by Lemma~\ref{lem:eq})} \\
& = E_{\bigcap_{i = 1}^{\ell}(O_i)_=} && \text{(by Lemma~\ref{lem:join})}.
\qedhere
\end{align*}
\end{proof}

\subsection{Actions on definable equivalence relations}
\label{sect:action}
The key in our analysis of the behaviour of canonical functions of $\bJ^<(k)$ is to understand their actions on ${\mathcal E}(\bJ^<(k))$, the
definable equivalence relations in $\bJ^<(k)$.

\begin{defi}
For $f \colon \qup{k} \to \qup{k}$ we define a function
${\mathcal B}_f \colon {\mathcal P}([k]) \to {\mathcal P}([k])$ by setting
\begin{align*}
\mathcal B_f(I):= \{j\in[k]:\text{ for every }(a,b)\in E_I,\ f(a)_j=f(b)_j\}   .
\end{align*}
\end{defi}

\noindent
For example, note that ${\mathcal B}_f(\emptyset) = \emptyset$ if and only if
$f(E_{\emptyset}) \subseteq E_{I}$ implies that $I=\emptyset$.
In the proofs, the following notation will be helpful. For each $j\in[k]$, define
$$
K_f(j):=\{(a,b)\in(\qup{k})^2\mid f(a)_j=f(b)_j\}.
$$
Note that $K_f(j)$ is an equivalence relation, and for every $H\subseteq[k]$ we have
$$
j\in{\mathcal B}_f(H)
\quad\Longleftrightarrow\quad E_H\subseteq K_f(j).
$$

\begin{lem}\label{lem:cap}
Let $f \colon \qup{k} \to \qup{k}$
and $I,J \subseteq [k]$. Then
\begin{align}
{\mathcal B}_f(I \cap J) & = {\mathcal B}_f(I) \cap {\mathcal B}_f(J) \label{eq:cap1} \\
\text{ and } \quad
{\mathcal B}_f(I) & = \bigcap_{i \in [k]  \setminus I} {\mathcal B}_f([k] \setminus \{i\}). \label{eq:cap2}
\end{align}
\end{lem}
\begin{proof}

Since $K_f(j)$ is an equivalence relation and
$E_{I\cap J}=E_I\vee E_J$ by Lemma~\ref{lem:join}, we have
\begin{align*}
j\in{\mathcal B}_f(I\cap J)
&\Longleftrightarrow E_{I\cap J}\subseteq K_f(j)\\
&\Longleftrightarrow E_I\subseteq K_f(j)
   \text{ and }E_J\subseteq K_f(j)\\
&\Longleftrightarrow
j\in{\mathcal B}_f(I)\cap{\mathcal B}_f(J).
\end{align*}
This proves~\eqref{eq:cap1}.
For $I\neq[k]$, equation~\eqref{eq:cap2} follows by repeatedly
applying~\eqref{eq:cap1} to
$$
I=\bigcap_{i\in[k]\setminus I}([k]\setminus\{i\}).
$$
For $I=[k]$, both sides of~\eqref{eq:cap2} equal $[k]$,
since $E_{[k]}$ is equality and we take the intersection of
an empty family of subsets of $[k]$ to be $[k]$.
\end{proof}

\begin{lem}\label{lem:=}
Let $f \in \Can(\bJ^<(k))$ and
$a,b \in \qup{k}$.
Then
${\mathcal B}_f((a,b)_=) = (f(a),f(b))_=$.
\end{lem}
\begin{proof}
Put $I:=(a,b)_=$, and let $O$ be the orbit of $(a,b)$
under $\Aut(\bJ^<(k))$.
Since $(a,b)\in E_I$, the definition of ${\mathcal B}_f$ gives
${\mathcal B}_f(I)\subseteq(f(a),f(b))_=$.
For the converse inclusion,
let $j\in(f(a),f(b))_=$.
Since $f$ is canonical and $f(a)_j=f(b)_j$, we have
$f(u)_j=f(v)_j$ for every $(u,v)\in O$.
Thus, $O\subseteq K_f(j)$, and hence $\Eq(O)\subseteq K_f(j)$.
By Lemma~\ref{lem:eq}, we obtain $E_I=\Eq(O)\subseteq K_f(j)$.
Therefore, $j\in{\mathcal B}_f(I)$ by definition.
This proves the converse inclusion and hence the assertion.
\end{proof}

\begin{defi}\label{def:L}
For $a,b \in \qup{k}$, we define $L(a,b) := \{ i \in [k] \mid a_i = b_j \text{ for some } j \in [k] \}$.
\end{defi}

\noindent
Note that $|L(a,b)| = |a \cap b| = |L(b,a)|$.

\begin{lem}\label{lem:L}
Let $f \in \Can(\bJ^<(k))$
and $a,b \in \qup{k}$.
Then $L(f(a),f(b)) \subseteq {\mathcal B}_f (L(a,b))$.
\end{lem}
\begin{proof}
Choose $c \in \qup{k}$ such that $a_i = c_i$ if and only if $i \in L(a,b)$;
we may choose $c$ such that $\max_{i \in [k]} \{|c_i - a_i |\}$ is arbitrarily small, so that
$(a,b)$ and $(c,b)$ lie in the same orbit.
Let $i \in L(f(a),f(b))$, i.e., there exists $j \in [k]$ such that
$f(a)_i = f(b)_j$.
The canonicity of $f$ implies that $f(c)_i=f(b)_j$, and hence
$f(a)_i = f(c)_i$.
By the choice of $c$, we have $(a,c)_= = L(a,b)$.
Hence, Lemma~\ref{lem:=} gives
\begin{align*}
i\in(f(a),f(c))_=
& ={\mathcal B}_f((a,c)_=)
 ={\mathcal B}_f(L(a,b)). &&
\qedhere
\end{align*}
\end{proof}

\begin{cor}\label{cor:S0}
Let $f \in \Can(\bJ^<(k))$
be such that ${\mathcal B}_f(\emptyset) = \emptyset$. Then $f$ preserves $S_0$.
\end{cor}
\begin{proof}
Let $(a,b) \in S_0$. Then $L(a,b) = \emptyset$. We have
\begin{align*}
L(f(a),f(b)) & \subseteq {\mathcal B}_f(L(a,b)) && \text{(Lemma~\ref{lem:L})} \\
& = {\mathcal B}_f(\emptyset) = \emptyset && \text{(by assumption)}
\end{align*}
 and hence $(f(a),f(b)) \in S_0$.
\end{proof}

\subsection{Permutational maps}
In this section we introduce an important property that a function $f$ which is canonical with respect to $\bJ^<(k)$ might have; this property plays an important role in our classification of canonical functions in Section~\ref{sect:canonical}.
A function $f \in \Can(\bJ^<(k))$ is called \emph{permutational} if there exists a permutation $\pi \in \Sym([k])$ such that ${\mathcal B}_f(I) = \pi(I)$ for all $I \subseteq [k]$.

\begin{exa}
Let $\alpha \in \Aut({\mathbb Q};<)$, and let $\beta \colon {\mathbb Q} \to {\mathbb Q}$ be an order-reversing bijection, e.g., $x \mapsto -x$.
Then the operations
\[
f_\alpha \colon (a_1,\dots,a_k) \mapsto (\alpha(a_1),\dots,\alpha(a_k))
\quad\text{and}\quad
f_\beta \colon (a_1,\dots,a_k) \mapsto (\beta(a_k),\dots,\beta(a_1))
\]
are automorphisms of $\bJ(k)$ that are canonical with respect to $\bJ^<(k)$.
They are permutational, since ${\mathcal B}_{f_\alpha}(I) = I$ and ${\mathcal B}_{f_\beta}(I) = \{k+1-i \mid i \in I\}$ for all $I \subseteq [k]$.
\end{exa}

\noindent
We will show in this section that any permutational operation which is canonical with respect to $\bJ^<(k)$ is either from $\overline{\Aut(\bJ(k))}$, or generates an operation which is not permutational (Lemma~\ref{lem:perm-emb}). We then prove in Section~\ref{sect:non-perm} that non-permutational canonical operations generate operations that will later in Section~\ref{sect:canonical} be highly useful when analysing operations that are canonical with respect to $\bJ^<(k)$.

Note that if $f$ is permutational, then ${\mathcal B}_f(\emptyset) = \emptyset$.
Lemma~\ref{lem:cap} implies that $f \in \Can(\bJ^<(k))$ is permutational if and only if
there exists a permutation $\pi \in \Sym([k])$ such that ${\mathcal B}_f([k] \setminus \{i\}) = [k] \setminus \{\pi(i)\}$ for every $i \in [k]$.

\begin{lem}\label{lem:non-exp}
Suppose that $f \in \Can(\bJ^<(k))$ is permutational.
Then $|f(a) \cap f(b) | \leq |a \cap b|$ for all $a,b \in \qup{k}$.
\end{lem}
\begin{proof}
Let $a,b \in \qup{k}$.
Then $L(f(a),f(b)) \subseteq {\mathcal B}_f(L(a,b)) = \pi(L(a,b))$ by Lemma~\ref{lem:L}.
Hence, $|f(a) \cap f(b)|  = |L(f(a),f(b))| \leq |\pi(L(a,b))| = |L(a,b)| = |a \cap b|$.
\end{proof}

\begin{cor}\label{cor:step-down}
	Suppose that $f \in \Can(\bJ^<(k))$ is permutational and satisfies $f(S_n)\cap S_{\leq n-1}\neq \emptyset$ for some $n \in \{1,\dots,k\}$. Then there exists some $g\in \overline{\langle \{f\} \cup \Aut(\bJ(k)) \rangle}$ such that $g(S_n)\subseteq S_{\leq n-1}$.
\end{cor}

\begin{proof}
Lemma~\ref{lem:non-exp} implies that $f$, and thus also any map generated by $f$ and $\Aut(\bJ(k))$, preserves the relation $S_{\leq n-1}$. Then it follows from an easy induction that any finite subset of $S_n$ can be mapped to $S_{\leq n-1}$ by some element in $\langle \{f\} \cup \Aut(\bJ(k)) \rangle$. The statement of the corollary then follows from a standard compactness argument.
\end{proof}

\begin{lem}\label{lem:perm-emb}
Let $f \in \Can(\bJ^<(k))$ be permutational. Then either $f \in \overline{\Aut(\bJ(k))}$
or there exists some $g\in \overline{\langle \{f\} \cup \Aut(\bJ(k)) \rangle} \cap \Can(\bJ^<(k))$ which is not permutational.
\end{lem}
\begin{proof}
	Let us first assume that $f(S_{k-1})\cap S_{\leq k-2} \neq \emptyset$.
    Then by Corollary~\ref{cor:step-down} we can find some $g\in \overline{\langle \{f\} \cup \Aut(\bJ(k)) \rangle}$ such that $g(S_{k-1})\subseteq S_{\leq k-2}$.
    Since $\bJ^<(k)$ has the Ramsey property we can also find such a $g$ which is canonical with respect to $\bJ^<(k)$ (Lemma~\ref{lem:canon}).
	Note, however, that in this case $g$ cannot be permutational: to see this, pick $(a,b) \in E_{\{1,\dots,k-1\}}$ with $a \neq b$.
	If $g$ is permutational, then $(g(a),g(b)) \in E_{[k] \setminus \{j\}}$ for some $j \in [k]$,
	and hence $(g(a),g(b)) \notin S_{\leq k-2}$. Since $(a,b) \in S_{k-1}$, this contradicts our assumptions.

	On the other hand, if $f(S_{k-1})\cap S_{\leq k-2}= \emptyset$ then by Lemma~\ref{lem:non-exp} it follows that $f$ preserves $S_{k-1}$. By Corollary~\ref{cor:S0} we know that $f$ also preserves $S_0$. Since all the relations $S_1,\dots,S_{k-2}$ are primitively positively definable from $S_{k-1}$ and $S_0$ (Lemma~\ref{lem:Si}), we have that $f$ also preserves $S_1$, and thus by Corollary~\ref{cor:j-is-mccore} we have $f\in \overline{\Aut(\bJ(k))}$.
\end{proof}

\subsection{Non-permutational maps}
\label{sect:non-perm}
We first show that a non-permutational $f$ with ${\mathcal B}_f(\emptyset)=\emptyset$ decreases the number of coordinate agreements for some pair of tuples, i.e.,
$|(f(a),f(b))_=| < |(a,b)_=|$ (Lemma~\ref{lem:permutational} together with Lemma~\ref{lem:=}).
We then use compositions with automorphisms of $\bJ(k)$ to turn
this into a method for reducing intersections between images
of tuples, unless those images coincide (Corollary~\ref{cor:non-perm}).

\begin{lem}\label{lem:permutational}
Let $f \in \Can(\bJ^<(k))$
be such that ${\mathcal B}_f(\emptyset) = \emptyset$ and $f$ is not permutational.
Then
$|{\mathcal B}_f([k] \setminus \{i\})| \leq k-2$ for some $i \in [k]$.
\end{lem}
\begin{proof}
For every $i\in[k]$, put
$$
D_i:=[k]\setminus{\mathcal B}_f([k]\setminus\{i\}).
$$
By Lemma~\ref{lem:cap} and the assumption
${\mathcal B}_f(\emptyset)=\emptyset$, we have
$$
\bigcup_{i\in[k]}D_i
=[k]\setminus\bigcap_{i\in[k]}
{\mathcal B}_f([k]\setminus\{i\})
=[k].
$$
If $|D_i|\leq1$ for every $i\in[k]$, then these $k$ sets
must be pairwise distinct singletons.
Hence, there exists a permutation $\pi\in\Sym([k])$
such that $D_i=\{\pi(i)\}$ for every $i\in[k]$.
Equation~\eqref{eq:cap2} then gives for every $I\subseteq[k]$
$$
{\mathcal B}_f(I)
=[k]\setminus\bigcup_{i\in[k]\setminus I}D_i
=\pi(I).
$$
Thus, $f$ is permutational, contrary to our assumption.
Therefore, $|D_i|\geq2$ for some $i\in[k]$, and hence
$|{\mathcal B}_f([k]\setminus\{i\})|\leq k-2$.
\end{proof}

\begin{lem}\label{lem:step}
Let $f \in \Can(\bJ^<(k))$
be such that ${\mathcal B}_f(\emptyset) = \emptyset$ and $|{\mathcal B}_f([k] \setminus \{i\})| \leq k-2$ for some $i \in [k]$.
Then for all $a,b \in \qup{k}$
there exists $g \in \langle \{f \} \cup \Aut(\bJ(k)) \rangle$ such that
$g(a) = g(b)$ or $g(a) \cap g(b) = \emptyset$.
\end{lem}
\begin{proof}
We first show that there exist distinct
$i_1,\dots,i_{k-1} \in [k]$ such that
for all $\ell \in [k-1]$ we have $|\bigcap_{j \in [\ell]} {\mathcal B}_f([k] \setminus \{i_j\})| \leq k-\ell-1$.
We choose $i_1 := i$. Suppose inductively that we have already defined distinct $i_1,\dots,i_\ell$ for $\ell \in [k-2]$
such that $|I| \leq k-\ell-1$ for $I := \bigcap_{j \in [\ell]} {\mathcal B}_f([k] \setminus \{i_j\})$. If $|I| \leq k-\ell-2$, then $i_{\ell+1}$ can be chosen to be any element of $[k] \setminus \{i_1,\dots,i_{\ell}\}$, and we
get that
$$|\bigcap_{j \in [\ell+1]} {\mathcal B}_f([k] \setminus \{i_j\})| = | I \cap {\mathcal B}_f([k] \setminus \{i_{\ell+1}\}) | \leq k - \ell - 2 = k - (\ell+1) - 1.$$
Otherwise, there exists some $s \in I$.
Observe that
\begin{align*}
    I\cap \bigcap_{j \in [k]\setminus \{i_1,\dots,i_{\ell}\}} {\mathcal B}_f([k] \setminus \{j\}) & =\bigcap_{j \in [k]} {\mathcal B}_f([k] \setminus \{j\}) \\
    & =
    {\mathcal B}_f(\emptyset) && \text{(by~\eqref{eq:cap1})} \\
    & =\emptyset && \text{(by assumption).}
\end{align*}
Thus, we can choose $i_{\ell+1}$ to be such that $s \notin {\mathcal B}_f([k] \setminus \{i_{\ell+1}\})$.
Then
$$|\bigcap_{j \in [\ell+1]} {\mathcal B}_f([k] \setminus \{i_j\})| = |I \cap {\mathcal B}_f([k] \setminus \{i_{\ell+1}\})| \leq (k-\ell-1) - 1 = k - (\ell + 1) - 1$$
 as well.

Let $a,b \in \qup{k}$. If there exists
$g \in \langle \{f \} \cup \Aut(\bJ(k)) \rangle$
such that $g(a) = g(b)$ then we are done. So we suppose that $g(a) \neq g(b)$ for all $g \in \langle \{f \} \cup \Aut(\bJ(k)) \rangle$. Choose $g \in  \langle \{f \} \cup \Aut(\bJ(k)) \rangle$
such that $\ell := |g(a) \cap g(b)|$ is minimal.
Suppose for contradiction that $\ell \geq 1$.
Choose
$c_1,\dots,c_k,d_1,\dots,d_k \in {\mathbb Q}$ such that
$$c_1 \leq d_1 < c_2 \leq d_2 < \cdots < c_k \leq d_k$$
and $c_i < d_i$ if and only if $i \in \{i_1,\dots,i_{k-\ell} \}$. Note that in this case
$|c \cap d| = \ell$.
Hence,  there exists $\alpha \in \Aut(\bJ(k))$
such that $(c,d) = \alpha g(a,b)$.
By construction, $L(c,d)=(c,d)_=$ (Definition~\ref{def:L}).
Hence, Lemmas~\ref{lem:L} and~\ref{lem:=} imply that
$$
L(f(c),f(d))
\subseteq{\mathcal B}_f(L(c,d))
={\mathcal B}_f((c,d)_=)
=(f(c),f(d))_=.
$$
Since the reverse inclusion always holds, we obtain
$$
|f(c)\cap f(d)|
=|(f(c),f(d))_=|
=|{\mathcal B}_f((c,d)_=)|.
$$
Moreover,
\begin{equation*}
|{\mathcal B}_f((c,d)_=)|
= \Bigl|\bigcap_{j \in [k-\ell]} {\mathcal B}_f([k] \setminus \{i_j\})\Bigr|
\leq k - (k-\ell) - 1 = \ell - 1.
\tag{Lemma~\ref{lem:cap}}
\end{equation*}
This implies that $|f(c) \cap f(d)| \leq \ell -1$, contradicting the minimality of $\ell$.
Hence, there exists $g \in  \langle \{f \} \cup \Aut(\bJ(k)) \rangle$ with $g(a) \cap g(b) = \emptyset$, which concludes the proof.
\end{proof}

\noindent
The combination of the previous two lemmata implies the following.

\begin{cor}\label{cor:non-perm}
Let $f \in \Can(\bJ^<(k))$
be such that ${\mathcal B}_f(\emptyset) = \emptyset$ and $f$ is not permutational.
Then for all $a,b\in \qup{k}$ there exists $g \in \langle \{f \} \cup \Aut(\bJ(k)) \rangle$ such that
\begin{align*}
& g(a) = g(b) \\
\text{ or }  & g(a) \cap g(b) = \emptyset.
\end{align*}
\end{cor}

\subsection{Canonical functions}
\label{sect:canonical}
In this section we present an analysis of canonical functions  with respect to $\Aut(\bJ^<(k))$.

\begin{thm}\label{thm:canonical}
Let $f \in \Can(\bJ^<(k))$.
Then at least one of the following holds.
\begin{enumerate}
\item $f \in \overline{\Aut(\bJ(k))}$.
\item there exists a non-empty $F \subseteq {\mathbb Q}$ such that for every $a \in \qup{k}$ we have $F \subseteq f(a)$.
\item
$\{f\} \cup \Aut(\bJ(k))$ locally generates
an operation $g$ such that any two distinct elements in the image of $g$ are disjoint.
\end{enumerate}
\end{thm}

\noindent
The proof of this theorem rests on the following two lemmata.

\begin{lem}\label{lem:expl}
Let $f \in \Can(\bJ^{<}(k)) \setminus \overline{\Aut(\bJ(k))}$ be such that ${\mathcal B}_f(\emptyset) = \emptyset$.
Then for all $a,b \in \qup{k}$ there exists $g \in \overline{\langle \{f\} \cup \Aut(\bJ(k)) \rangle}$
 such that
$g(a) = g(b)$ or $g(a) \cap g(b) = \emptyset$.
\end{lem}
\begin{proof}
If $f$ is not permutational, then
the statement follows from Corollary~\ref{cor:non-perm}.
Otherwise,
 Lemma~\ref{lem:perm-emb} implies that there exists some $h \in \overline{\langle\{f\} \cup \Aut(\bJ(k))\rangle}\cap \Can(\bJ^{<}(k))$ such that $h$ is not permutational. Note that $h$ still preserves the relation $S_0$, and thus ${\mathcal B}_{h}(\emptyset) = \emptyset$.
Then the statement of the lemma follows again from Corollary~\ref{cor:non-perm}.
\end{proof}

\noindent
Lemma~\ref{lem:expl} has the following global version.

\begin{lem}\label{lem:emb-disj-eq}
Let $f \in \Can(\bJ^{<}(k)) \setminus \overline{\Aut(\bJ(k))}$ be such that ${\mathcal B}_f(\emptyset) = \emptyset$.
Then $\{f\} \cup \Aut(\bJ(k))$ locally generates a function $g$ such that $g((\qup{k})^2) \subseteq S_0 \cup S_k$.
\end{lem}
\begin{proof}
Let $F$ be a finite subset of $\qup{k}$. We have to show that there exists $h \in \langle \{f\} \cup \Aut(\bJ(k)) \rangle$ such that $h(F^2) \subseteq S_0 \cup S_k$.
Let $(a_1,b_1),\dots,(a_n,b_n)$ be an enumeration of $F^2$. We show by induction on $i \in [n]$  that there exists $h_i \in \langle \{f\} \cup \Aut(\bJ(k)) \rangle$ such that
$$h_i \circ \cdots \circ h_1(a_i,b_i) \in S_0 \cup S_k.$$
For $i = 1$, the statement is immediate from Lemma~\ref{lem:expl} applied to $(a,b) = (a_1,b_1)$.
For $i \geq 2$, the statement follows by applying Lemma~\ref{lem:expl} to $(a,b) := h_{i-1} \circ \cdots \circ h_1(a_i,b_i)$.

Then $h_n \circ \cdots \circ h_1$ maps $(a_1,b_1),\dots,(a_n,b_n)$ into $S_0 \cup S_k$, because $h_1,\dots,h_n$ preserve $S_0$ and $S_k$. The claim now follows from local closure and a standard compactness argument.
\end{proof}

\begin{proof}[Proof of Theorem~\ref{thm:canonical}]
Suppose that
$f \notin \overline{\Aut(\bJ(k))}$.
First consider the case that
${\mathcal B}_f(\emptyset) \neq \emptyset$.
Consider any $a \in \qup{k}$ and define $F := \{f(a)_i \mid i \in {\mathcal B}_f(\emptyset)\}$. Identifying the
elements of $\qup{k}$ with ${{\mathbb N} \choose k}$ as explained in Section~\ref{sect:RamseyExp}, we have that $F \subseteq f(b)$ for every $b \in \qup{k}$
by the definition of ${\mathcal B}_f(\emptyset)$
and we are done.

Otherwise, ${\mathcal B}_f(\emptyset) = \emptyset$.
Then $f$ preserves $S_0$ by Corollary~\ref{cor:S0}.
By Lemma~\ref{lem:emb-disj-eq},
$\{f\} \cup \Aut(\bJ(k))$
locally generates some function $g$ such that $g((\qup{k})^2) \subseteq S_0 \cup S_k$.
Since $g$ preserves $S_0$, the image of
$g$ is infinite, and consists of infinitely many pairwise disjoint sets.
\end{proof}

\subsection{Range-rigid functions}
Let $G$ be a permutation group on a set $A$.
A function $g \colon A \to A$ is called \emph{range-rigid with respect to $G$} if
 for all $n \in \mathbb N$, all orbits of the componentwise action of $G$ on $A^n$ that have a non-empty intersection with $g(A)^n$ are preserved by $g$.
 We will use the following theorem.

\begin{thmC}[{\cite[proof of Theorem~4]{MottetPinskerCores}}]\label{thm:range-rigid}
Let $\bB$ be a first-order reduct of a countable $\omega$-categorical homogeneous Ramsey structure $\bA$ and let $\bC$ be the model-complete core of $\bB$. Then
\begin{itemize}
\item there exists $f \in \End(\bB)$ which is range-rigid and canonical
with respect to $\Aut(\bA)$
 such that the structure induced by the image of $f$ in $\bA$ has the same age as a
homogeneous  Ramsey substructure $\bA'$ of $\bA$, and
\item $\bC \cong \bB[A']$.
\end{itemize}
\end{thmC}

\noindent
We use this to prove the following lemma, which is useful for classifying model-complete cores of first-order reducts of $\omega$-categorical homogeneous Ramsey structures.

\begin{lem}\label{lem:use-rr}
Let $\bB$ be a first-order reduct of a countable $\omega$-categorical homogeneous Ramsey structure $\bA$, and suppose $\bB$  is not a model-complete core. Then there exists an $f \in \End(\bB)$ which is not a self-embedding of $\bA$ and which is
range-rigid and
canonical with respect to $\Aut(\bA)$.
\end{lem}
\begin{proof}
Let $\bC$ be the model-complete core of $\bB$.
Let $f \in \End(\bB)$ be the map from Theorem~\ref{thm:range-rigid}
which is such that the image of $f$ has the same age as the homogeneous substructure $\bA'$ of $\bA$.
If $f$ is not a self-embedding of $\bA$ then we are done. Otherwise, $\Age(\bA') = \Age(\bA)$ and hence the countable homogeneous structures $\bA'$ and $\bA$ are isomorphic. Since $\bC=\bB[A']$ by Theorem~\ref{thm:range-rigid},
we have that $\bC$ is isomorphic to $\bB$.
Thus, $\bB \cong \bC$ is a model-complete core, contrary to the assumption.
\end{proof}

\noindent
The next lemma illustrates the use of range-rigidity.

\begin{lem}\label{lem:emb}
Let $f \in \overline{\Aut(\bJ(k))}$ be canonical and range-rigid with respect to $\Aut(\bJ^<(k))$. Then $f \in \End(\bJ^<(k))$.
\end{lem}
\begin{proof}
Let $\xi$ be the isomorphism between $\End({{\mathbb N} \choose k};S_0,S_1)=\overline{\Aut(\bJ(k))}$
and $\End({\mathbb Q};\neq)$ from Corollary~\ref{cor:j-is-mccore} (stated there for
${\mathbb N}$ instead of
${\mathbb Q}$, which is clearly irrelevant).
Concretely, $h := \xi(f)$ maps $x \in {\mathbb Q}$ to the unique element of $f(u) \cap f(v)$, for any $u,v \in \qup{k}$ with $u \cap v = \{x\}$, and $f(a)$ is the increasing enumeration of $\{h(a_1),\dots,h(a_k)\}$ for every $a \in \qup{k}$.
Since $f$ is canonical with respect to $\bJ^<(k)$, the map $h$ is canonical with respect to $({\mathbb Q};<)$: for $x<y$ and $u,v,u',v' \in \qup{k}$ with $u \cap v = \{x\}$ and $u' \cap v' = \{y\}$, whether $h(x)<h(y)$ is determined by the orbit of $(f(u),f(v),f(u'),f(v'))$, hence by the orbit of $(u,v,u',v')$, and for any other pair $x'<y'$ the witnesses can be chosen in the same orbit (apply some $\alpha \in \Aut({\mathbb Q};<)$ with $\alpha(x)=x'$ and $\alpha(y)=y'$).
Hence, $h$ is either order preserving or order reversing.
In the first case, $f \in \End(\bJ^<(k))$;
the second case is impossible by the assumption that $f$ is range-rigid with respect to $\bJ^<(k)$,  because then $f$ does not preserve the orbit of all pairs $(u,v)$ with $u_k < v_1$, although this orbit intersects the image of $f$ (since $h({\mathbb Q})$ is infinite).
\end{proof}

\subsection{Model-complete Core Classification}
The following is the main result of this section.

\begin{thm}\label{thm:main}
For $k \in {\mathbb N}$, let $\bB$ be a first-order reduct of $\bJ(k)$ and let $\bC$ be the model-complete core of $\bB$.
Then $\bC$ is first-order bi-definable with $\bJ(\ell)$, for some $\ell \leq k$.
\end{thm}

\begin{proof}
We prove the theorem by induction on $k \in {\mathbb N}$.
For $k = 0$ there is nothing to prove.
The statement is trivial for $k=1$ (the model-complete core of a structure with a highly transitive automorphism group is either  highly transitive or has only one element).
From now on we assume that $k \geq 2$ and that the statement is true for all smaller values of $k$.

 By Theorem~\ref{thm:fo-reducts},
$\bB$ is first-order bi-definable with $\bJ(1)$ or with $\bJ(k)$.
We now use Lemma~\ref{lem:use-rr}.
If $\bB$ is a model-complete core then there is nothing to be shown. Otherwise, there exists an $f \in \End(\bB) \setminus \End(\bJ^<(k))$ which is canonical and range-rigid with respect to $\Aut(\bJ^<(k))$.
By Lemma~\ref{lem:emb}, we have that
$f \in \End(\bB) \setminus \overline{\Aut(\bJ(k))}$.
We now use Theorem~\ref{thm:canonical}.

First consider the case that $\{f \} \cup \Aut(\bJ(k))$ locally generates an operation $g$ such that
$g(\qup{k})$ consists of infinitely many pairwise disjoint sets. Note that $\bB$ and  $\bB[g(B)]$ are homomorphically equivalent,
because $g \in \End(\bB)$. So $\bB$ and $\bB[g(B)]$ have the same model-complete core. However, $\Aut(\bB[g(B)]) = \Sym(g(B))$, because every permutation of a family of pairwise disjoint $k$-sets is induced by a permutation of ${\mathbb N}$.
It follows that $\Aut(\bC) = \Sym(C)$,
and hence is bi-definable with $\bJ(0)$ (in case that $|C|=1$) or with $\bJ(1)$
 (in case that $C$ is infinite).

Since $f \notin \overline{\Aut(\bJ(k))}$
by Theorem~\ref{thm:canonical}
we are left with the case that there exists some finite non-empty $F \subseteq {\mathbb Q}$
such that $F \subseteq f(a)$ for every
$a \in \qup{k}$. Since $\bB[f(B)]$ is homomorphically equivalent to $\bB$,
and $f(B) \subseteq B_F := \{a \in \qup{k} \mid F \subseteq a\}$, we have that $\bB[B_F]$ is also homomorphically equivalent to $\bB$. The structure
$\bJ(k)[B_F]$ is isomorphic to
$\bJ(k-|F|)$ via the bijection $a \mapsto a \setminus F$. This implies that
$\bB[B_F]$ is isomorphic to a first-order reduct of
$\bJ(k-|F|)$. Thus, by the induction hypothesis the model-complete core of $\bB[B_F]$, which is also $\bC$, is bi-definable with $\bJ(\ell)$ for
$\ell \leq k-|F| < k$.
\end{proof}

\section{Complexity Classification}
\label{sect:compl}
In this section we combine all our results on first-order reducts $\bC$ of $\bJ(k)$ to obtain a complexity classification for $\Csp(\bC)$.

\begin{thm}\label{thm:main-complexity}
Let $k \in {\mathbb N}$,
let $\bB$ be a first-order reduct of $\bJ(k)$,
and let $\bC$ be the model-complete core of $\bB$.
Then exactly one of the following cases applies.
\begin{enumerate}
\item $\bC$ has just one element.
\item $\End(\bC) = \End(C;\neq)$, $C$ is countably infinite, and $\bC$ has a binary injective polymorphism.
\item $\bC$ interprets primitively positively all finite structures.
\end{enumerate}
In the first two cases, for every reduct $\bB'$ of $\bB$ with finite relational signature,
$\Csp(\bB')$ is in P.
In the final case, there exists a reduct with finite relational signature $\bB'$ of $\bB$ such that $\Csp(\bB')$ is NP-complete.
\end{thm}
\begin{proof}
Theorem~\ref{thm:main} implies that $\bC$ is bi-definable with $\bJ(\ell)$, for some $\ell \leq k$. If $\ell = 0$, then $\bC$ has just one element and the statement is clear.
If $\ell=1$, then $\End(\bC) = \End(C;\neq)$, and the statement is well-known (\cite{ecsps}; also see~\cite[Theorems~7.5.1 and~7.5.2]{Book}).
If $\ell \geq 2$, then
Theorem~\ref{thm:J-stbb} implies
that $\bC$ interprets all finite structures primitively positively.
Hence, $\bB$ has a reduct $\bB'$ with finite relational signature whose $\Csp$ is NP-hard.
The containment in NP follows
e.g.\ from Remark~\ref{rem:inNP}.
It is easy to see that the three cases are mutually exclusive.
\end{proof}

\begin{cor}\label{cor:class}
Let $\bB$ be a countably infinite structure preserved by some primitive action of $\Sym({\mathbb N})$. Then either
$\bB$ primitively positively constructs all finite structures, or
$\Pol(\bB)$ contains ternary operations $f_1,f_2,f_3,f_4$ that satisfy the following identities for all $x,y \in B$
\begin{align*}
f_2(y,x,x) & = f_3 (y,x,x) = f_4(y,x,x) \\
f_1(y,x,x) & = f_3(x,y,x) = f_4(x,y,x) \\
f_1(x,y,x) & = f_2(x,y,x) = f_4(x,x,y) \\
f_1(x,x,y) & = f_2(x,x,y) = f_3(x,x,y).
\end{align*}
\end{cor}
\begin{proof}
The structure $\bB$ is a first-order reduct of $\bJ(k)$, for some $k \in {\mathbb N}$, by Corollary~\ref{cor:Jk}.
If the model-complete core $\bC$ of $\bB$ interprets primitively positively all finite structures, then $\bB$ primitively positively constructs all finite structures (see~\cite{wonderland}). Otherwise,
Theorem~\ref{thm:main-complexity}
implies that $\bC$ has just one element,
or $\End(\bC) = \End(C;\neq)$ and $\bC$ has a binary injective polymorphism $f$.
In the first case, $\bB$ has a constant polymorphism as well, and a constant ternary polymorphism $g$ clearly satisfies the given identities for $f_1=f_2=f_3=f_4 := g$.
 In the second case,
the existence of  $f_1,f_2,f_3,f_4 \in \Pol(\bC)$ satisfying the given identities
 has been shown in~\cite[Theorem~1.7]{RydvalDescr}; since $\bB$ is homomorphically equivalent to $\bC$, it has such polymorphisms as well.

Conversely, if $\Pol(\bB)$ contains such polymorphisms, then it does not primitively positively construct the structure $(\{0,1\};\nae)$, because
the polymorphisms of this finite structure do not satisfy the given identities, and because having polymorphisms that satisfy these identities is  preserved by primitive positive constructions.
\end{proof}

\section{Strengthened Characterisation of Model-Complete Cores}
In this section we present a strong consequence about primitive positive definability in first-order reducts of $\bJ(k)$, for $k \geq 2$, which is not needed for the complexity classification but is of independent interest.
For $I\subseteq\{0,\dots,k\}$, write
$S_I:=\bigcup_{i\in I}S_i$.

\begin{thm}\label{thm:ppdef}
Let $\bB$ be a first-order reduct of $\bJ(k)$ for $k \geq 2$.
Then the following are equivalent.
\begin{enumerate}
\item $\bB$ is a model-complete core and
$\bB$ is interdefinable with $\bJ(k)$.
\item $S_0,S_1,\dots,S_{k-1}$ are primitively positively definable in $\bB$.
\item there are distinct $i,j \in \{0,\dots,k-1\}$
such that $S_i$ and $S_j$ are primitively positively definable in $\bB$.
\item there are non-empty subsets $I_1,\dots,I_{\ell}$ of $\{0,\dots,k-1\}$ such that $S_{I_1},\dots,S_{I_\ell}$ are primitively positively definable in $\bB$ and $I_1 \cap \cdots \cap I_{\ell} = \emptyset$.
\end{enumerate}
\end{thm}

\noindent
For the interesting implication $(4) \Rightarrow (1)$ of Theorem~\ref{thm:ppdef} we need the following strengthening of Theorem~\ref{thm:canonical}.

\begin{lem}\label{lem:behave-extra}
Let $f \in \Can(\bJ^<(k))$.
Then at least one of the following holds.
\begin{enumerate}
\item $f \in \overline{\Aut(\bJ(k))}$.
\item ${\mathcal B}_f(\emptyset) \neq \emptyset$ and for all $a,b \in \qup{k}$ there exists $g \in \langle \{f\} \cup \Aut(\bJ(k)) \rangle$ such that $(g(a),g(b)) \in S_{\geq k-1}$.
\item For all $a,b \in \qup{k}$ there is $g \in \langle \{f\} \cup \Aut(\bJ(k)) \rangle$ such that $(g(a),g(b)) \in S_{|{\mathcal B}_f(\emptyset)|} \cup S_k$.
\end{enumerate}
\end{lem}
\begin{proof}

Put $r:=|{\mathcal B}_f(\emptyset)|$ and
$M:=\langle\{f\}\cup\Aut(\bJ(k))\rangle$.
We may assume that $f\notin\overline{\Aut(\bJ(k))}$,
since otherwise the first statement holds.
If $r=0$, then Lemma~\ref{lem:expl} implies that for every
$a,b\in\qup{k}$ there exists a map in $\overline M$ which maps
$(a,b)$ into $S_0\cup S_k$.
Some element of $M$ agrees with this map on $\{a,b\}$, so the
third statement holds. If $r=k$, then $f$ is constant and the
third statement again holds. Hence, we may assume that $0<r<k$.

We first make an observation that will be used in both remaining
cases. Let $I\subseteq[k]$, and choose $c,d\in\qup{k}$ such that
$$c_1\leq d_1<c_2\leq d_2<\cdots<c_k\leq d_k$$
and $c_i=d_i$ if and only if $i\in I$.
Then $L(c,d)=(c,d)_= = I$, and Lemmas~\ref{lem:=}
and~\ref{lem:L} give
$${\mathcal B}_f(I)=(f(c),f(d))_=
\subseteq L(f(c),f(d))\subseteq{\mathcal B}_f(I).$$
Consequently, $|f(c)\cap f(d)|=|{\mathcal B}_f(I)|$.
Since $\Aut(\bJ(k))$ acts transitively on $S_{|I|}$, it follows
that for all $x,y\in\qup{k}$ with $|x\cap y|=|I|$ there exists
$\alpha\in\Aut(\bJ(k))$ such that
$$|f\alpha(x)\cap f\alpha(y)|=|{\mathcal B}_f(I)|.$$
For every $i\in[k]$, put
$$D_i:=[k]\setminus{\mathcal B}_f([k]\setminus\{i\}).$$
By Lemma~\ref{lem:cap}, for every $I\subseteq[k]$ we have
$${\mathcal B}_f(I)
=[k]\setminus\bigcup_{i\in[k]\setminus I}D_i,
\qquad
\bigcup_{i\in[k]}D_i=[k]\setminus {\mathcal B}_f(\emptyset).$$
First suppose that $|D_i|\leq1$ for every $i\in[k]$.
Since $|\bigcup_{i\in[k]}D_i|=k-r<k$, there exist distinct
$p,q\in[k]$ such that $|D_p\cup D_q|\leq1$.
Let $a,b\in\qup{k}$, and choose $g\in M$ such that
$\ell:=|g(a)\cap g(b)|$ is maximal.
Suppose for contradiction that $\ell\leq k-2$.
Choose $J\subseteq[k]$ of size $k-\ell$ containing $p,q$,
and put $I:=[k]\setminus J$.
Then $|I|=\ell$ and
$$|{\mathcal B}_f(I)|
=k-\left|\bigcup_{i\in J}D_i\right|
\geq\ell+1.$$
By the observation above, there exists
$\alpha\in\Aut(\bJ(k))$ such that
$$|f\alpha g(a)\cap f\alpha g(b)|\geq\ell+1,$$
contrary to the maximality of $\ell$, since $f\alpha g\in M$.
Thus, the second statement holds.

Now suppose that $|D_i|\geq2$ for some $i\in[k]$.
Let $a,b\in\qup{k}$. If some $g\in M$ satisfies $g(a)=g(b)$,
then we are done. Otherwise, choose $g\in M$ such that
$\ell:=|fg(a)\cap fg(b)|$ is minimal. Any two elements in the image of $f$ share at least $r$ elements, because $f(a)_i$, for $i \in {\mathcal B}_f(\emptyset)$, does not depend on $a$. Hence,  $r\leq\ell\leq k-1$.
Suppose for contradiction that $\ell>r$, and put $m:=k-\ell$.
Then $|\bigcup_{j\in[k]}D_j|=k-r>m$.
Starting with the index $i$, add indices whose sets $D_j$
contain an element outside the current union until that union
has at least $m+1$ elements. This is possible by the preceding
inequality and requires at most $m$ indices, since $|D_i|\geq2$.
Adding further indices if necessary, we obtain a set
$J\subseteq[k]$ with $|J|=m$ and
$|\bigcup_{j\in J}D_j|\geq m+1$.
Put $I:=[k]\setminus J$. Then $|I|=k-m=\ell$ and
$$|{\mathcal B}_f(I)|
=k-\left|\bigcup_{j\in J}D_j\right|
\leq k-m-1=\ell-1.$$
Applying the observation to $fg(a),fg(b)$, we obtain
$\alpha\in\Aut(\bJ(k))$ such that
$$|f\alpha fg(a)\cap f\alpha fg(b)|\leq\ell-1.$$
This contradicts the minimality of $\ell$, since $\alpha fg\in M$.
Hence, $\ell=r$, and $fg\in M$ witnesses the third statement.
\end{proof}

\begin{proof}[Proof of Theorem~\ref{thm:ppdef}]
We have already observed the implication
$(1) \Rightarrow (2)$ in the proof of Theorem~\ref{thm:J-stbb}.
The implications $(2) \Rightarrow (3)$ and $(3) \Rightarrow (4)$ are trivial.
Finally, we prove the implication $(4) \Rightarrow (1)$. Let us first assume that $\bB$ is a model-complete core. Then under condition $(4)$, at least one \(I_t\) is a nonempty proper subset of \(\{0,\dots,k-1\}\). Its relation \(S_{I_t}\) is not preserved by the full symmetric group. Therefore, Theorem~\ref{thm:fo-reducts} gives \(\operatorname{Aut}(\mathfrak B)=\operatorname{Aut}(\mathfrak J(k))\).  Otherwise, Lemma~\ref{lem:use-rr}  implies that  there exists an $f \in \End(\bB) \setminus \End(\bJ^<(k))$
 which is canonical
  and range-rigid with respect to $\Aut(\bJ^<(k))$. By Lemma~\ref{lem:emb} we even have that $f \in \End(\bB) \setminus \overline{\Aut(\bJ(k))}$. We will show that this is impossible.

Lemma~\ref{lem:behave-extra} implies that
\begin{enumerate}
\item[(i)] for all $a,b \in \qup{k}$
there exists $g \in \langle \{f\} \cup \Aut(\bJ(k)) \rangle$ such that
$(g(a),g(b)) \in S_{\geq k-1}$, or
\item[(ii)] for all $a,b \in \qup{k}$ there exists
$g \in \langle \{f\} \cup \Aut(\bJ(k)) \rangle$ such that
$(g(a),g(b)) \in S_{|{\mathcal B}_{f}(\emptyset)|} \cup S_k$.
\end{enumerate}
Let $I \subseteq \{0,1,\dots,k-1\}$ be non-empty such that
 $S_I$ is primitively positively definable in $\bB$.
 Then $g$ preserves $S_I$
and since $k \notin I$ we have $(g(a),g(b)) \notin S_k$ for any pair $(a,b) \in S_I$.
This shows that if $I_1,\dots,I_{\ell}$ are
such that
$S_{I_1},\dots,S_{I_\ell}$ are primitively positively definable in $\bB$, then
$n \in I_1 \cap \cdots \cap I_{\ell}$
where $n=k-1$ in the first case and $n = |{\mathcal B}_{f}(\emptyset)|$ in the second. This contradicts our assumption that $I_1 \cap \cdots \cap I_{\ell}=\emptyset$.
\end{proof}

\section{Polymorphism clones of first-order reducts of Johnson graphs}\label{sect:collapse}

    We say that an operation $f \colon A^n\rightarrow A$ is \emph{essentially unary} if there exists an index $i \leq n$ and a map $\alpha \colon A\rightarrow A$ such that $f(x_1,\dots,x_n)=\alpha(x_i)$ for all $x_1,\dots,x_n\in A$. In this section we show the following strengthening of Theorem~\ref{thm:J-stbb}.

\begin{thm}\label{thm:J-coll}
    Let $\bC$ be any model-complete core structure which is first-order interdefinable with $\bJ(k)$, for some $k \geq 2$. Then every polymorphism of $\bC$ is essentially unary.
\end{thm}

    \noindent
    Note that Theorem~\ref{thm:J-coll} indeed implies Theorem~\ref{thm:J-stbb}; for an explicit reference see Corollary~6.1.21 in~\cite{Book}. However, in our approach we cannot use this implication since our proof of Theorem~\ref{thm:J-coll} relies on Theorem~\ref{thm:J-stbb}.

    As a consequence of Theorem~\ref{thm:J-coll} we also obtain the following strengthening of Theorem~\ref{thm:main}.

\begin{cor}\label{main_pp}
    For $k \in {\mathbb N}$, let $\bB$ be a first-order reduct of $\bJ(k)$ and let $\bC$ be the model-complete core of $\bB$. Then one of the following holds.
\begin{enumerate}
\item $|C|=1$,
\item $\bC$ is first-order bi-definable with the infinite pure set,
\item $\bC$ is primitively positively bi-definable with $({{\mathbb N} \choose \ell};S_0,S_1)$,
for some $2\leq \ell \leq k$.
\end{enumerate}
\end{cor}

\begin{proof}
    By Theorem~\ref{thm:main} we know that $\bC$ is first-order bi-definable with $\bJ(\ell)$, for some $\ell \leq k$. Thus, if $\ell\in \{0,1\}$ then there is nothing to prove. Otherwise, by Theorem~\ref{thm:J-coll} and Corollary~\ref{cor:j-is-mccore} we know that up to renaming the elements of $C$ the structures $\bC$ and $
(\binom{\mathbb N}{\ell};S_0,S_1)$
have the same set of polymorphisms. The statement then follows from the fact that a relation is pp-definable in an $\omega$-categorical structure if and only if it is preserved by all of its polymorphisms~\cite{BodirskyNesetrilJLC}.
\end{proof}

    \noindent
    To show Theorem~\ref{thm:J-coll} we use the following result.

\begin{thmC}[{\cite[Corollary~3.7]{marimon2024minimal}}]
\label{essential_binary}
    Let $\bA$ be a countably infinite transitive $\omega$-categorical model-complete core all of whose binary polymorphisms are essentially unary. Then all polymorphisms of $\bA$ are essentially unary.
\end{thmC}

    \noindent
    The next lemma follows directly from the proof of Theorem~\ref{thm:J-stbb}.

\begin{lem}\label{candidate}
	Let $\bC$ be a model-complete core with $\Aut(\bC)=\Aut(\bJ(k))$ for some $k\geq 2$. Then for every binary polymorphism $f$ of $\bC$ we have
\begin{itemize}
\item $f(O,S_{k-j})\subseteq S_{k-j}$ for $j\in \{1,2\}$; or
\item $f(S_{k-j},O)\subseteq S_{k-j}$ for $j\in \{1,2\}$
\end{itemize}
    where $O=S_{k-1}\cup S_{k-2}$.
\end{lem}

\begin{proof}
    We know that the map $I \colon O \to \{0,1\}$ defined by
$$I(x,y) := \begin{cases} 1 & \text{if } S_{k-1}(x,y) \\
 0 & \text{otherwise}
 \end{cases}$$
	gives a pp-interpretation of $(\{0,1\}; \nae)$.  This implies that the polymorphisms of $\bC$ act on $\{0,1\}$ in a natural way and this action preserves the relation $\nae$. Since $S_{k-2}$ and $S_{k-1}$ are orbits of pairs of ${\Aut}(\bC)$, in particular they are pp-definable in $\bC$. This means that the aforementioned action also preserves the constants 0 and 1. Since all polymorphisms of $(\{0,1\}; \nae,\{0\},\{1\})$ are projections the statement of the lemma follows.
\end{proof}

\begin{proof}[Proof of Theorem~\ref{thm:J-coll}]
    By Theorem~\ref{essential_binary} it is enough to check that all binary polymorphisms of $\bC$ are essentially unary. Let $f$ be such a polymorphism. By applying Lemma~\ref{candidate} and by flipping the arguments of $f$ if necessary we can assume that
    \(f(S_{k-j},O)\subseteq S_{k-j}\)
    for $j\in \{1,2\}$. We show that $f(x_1,x_2)=\alpha(x_1)$ for some $\alpha\in \End(\bC)$.

\medskip

{\bf Claim 1.} $f(S_{k-1},\neq)\subseteq S_{k-1}$.

We show by induction on $\ell$ that $$f(S_{k-1},S_{\geq k-\ell}\,\cap\neq)\subseteq S_{k-1}.$$ For $\ell\in \{1,2\}$ the statement follows from our assumption on $f$. Now let us assume that $\ell\geq 3$ and that the statement holds for $\ell-1$. Consider some elements $u_1,u_2,v_1,v_2\in {{\mathbb N} \choose k}$ with $|u_1\cap u_2|=k-1$ and $|v_1\cap v_2|=k-\ell$. We need to show that $|f(u_1,v_1)\cap f(u_2,v_2)|=k-1$.
Let $n\coloneqq k^2-k+4=(k(k-1)+2)+2$ and let us pick some elements $u_3,\dots,u_n,v_3,\dots,v_n\in {{\mathbb N} \choose k}$ such that
\begin{itemize}
\item $u_1,\dots,u_n$ pairwise intersect each other in $u_1\cap u_2$ and
\item $v_3,\dots,v_n$ pairwise intersect each other in $(v_1\cap v_2)\cup \{a,b\}$ for some $a\in v_1\setminus v_2$ and $b\in v_2\setminus v_1$.
\end{itemize}
	Then for all $3\leq i<j\leq n$ we have $|u_i\cap u_j|=k-1$ and $|v_i\cap v_j|=k-\ell+2$.
    Therefore, by the induction hypothesis we can conclude that $|f(u_i,v_i)\cap f(u_j,v_j)|=k-1$ for all $3\leq i<j\leq n$.
    By Lemma~\ref{lem:intersect} we obtain that the elements $f(u_i,v_i): 3\leq i\leq n$ pairwise intersect each other in some $(k-1)$-element subset $r$.
    Now let us observe that $|u_1\cap u_i|=k-1$ and $|v_1\cap v_{i}|\geq k-\ell+1$ also holds for all $3\leq i\leq n$. Thus, by repeating the argument above it follows that $f(u_1,v_1)$ also must intersect $f(u_i,v_i)$  in $r$, for each $i \in \{3,\dots,n\}$.
    By a similar argument we obtain that $r\subseteq f(u_2,v_2)$. Therefore, either $f(u_1,v_1)\cap f(u_2,v_2)=r$ or $f(u_1,v_1)=f(u_2,v_2)$.
    However, we have seen in the proof of Theorem~\ref{thm:J-stbb} that the relation $\neq$ is pp-definable from the relations $S_0$ and $S_1$, thus it is also definable in $\bC$.
    Since polymorphisms preserve pp-definable relations it follows that $f(u_1,v_1)\neq f(u_2,v_2)$.
    Therefore, $f(u_1,v_1)\cap f(u_2,v_2)=r$, and thus $|f(u_1,v_1)\cap f(u_2,v_2)|=k-1$.

\medskip
{\bf Claim 2.}
	For all $r\in {\mathbb{N}\choose k-1}$ there exists a unique $\beta(r) \in {\mathbb{N}\choose k-1}$ such that if $r\subset u\in {\mathbb{N}\choose k}$ then $\beta(r)\subseteq f(u,v)$ for all $v\in {\mathbb{N}\choose k}$.

	Let $n\coloneqq k^2+2=(k(k-1)+2)+k$ and let us consider some subsets $u_1,\dots,u_n,v_1,\dots,v_n\in {{\mathbb N} \choose k}$ such that
\begin{itemize}
\item the pairwise intersection of $u_1,\dots,u_n$ is $r$, and
\item $v_1,\dots,v_n$ are pairwise disjoint.
\end{itemize}

	\noindent
	By Claim 1 we know that $|f(u_i,v_i)\cap f(u_j,v_j)|=k-1$ for all $1\leq i<j\leq n$, and therefore by Lemma~\ref{lem:intersect} we know that the elements $f(u_i,v_i): 1\leq i\leq n$ pairwise intersect each other in some $(k-1)$-element subset $\beta(r)$. We show that this choice of $\beta(r)$ suffices. Arbitrarily choose $u=u_0,v=v_0\in {\mathbb{N}\choose k}$  such that $r\subset u$. By our construction it follows that there exists at most one index $i\geq 1$ such that $u_i=u$ and at most $k$ indices $j\geq 1$ such that $v_j\cap v\neq \emptyset$. Therefore, there exists some $I\subset [n]$ of size at least $k^2-k+1$ such that
	\begin{itemize}
	\item the pairwise intersection of the sets $u_i$, for $i\in I\cup \{0\}$, is $r$ and
	\item the sets $v_i$, for $i\in I\cup \{0\}$,  are pairwise disjoint.
	\end{itemize}
By Claim 1 it follows that the pairwise intersection of the elements $f(u_i,v_i)$, for $i\in I\cup \{0\}$, is of size $k-1$. Now by using Lemma~\ref{lem:intersect} again we obtain that there exists a $(k-1)$-element subset $s$ such that $s\subset f(u_i,v_i)$ for all $i\in I\cup \{0\}$. This is only possible if $s=\beta(r)$.

    \medskip
{\bf Claim 3.} Let $\beta$ be as in the conclusion of Claim 2. Then for all $r,s\in {\mathbb{N}\choose k-1}$ such that $|r\cup s|=k$ we have that $|\beta(r)\cup \beta(s)|=k$.

	Let $u\coloneqq r\cup s$. Then by definition $\beta(r)\cup \beta(s)\subseteq f(u,u)$, and thus $|\beta(r)\cup \beta(s)|\leq k$.
Choose distinct $p,q\in{\mathbb N}\setminus u$, and put
$$v:=r\cup\{p\},\qquad w:=s\cup\{q\}.$$
Then $u\cap v=r$, $u\cap w=s$, and
$v\cap w=r\cap s$, so $|v\cap w|=k-2$.
    Thus, $\beta(r)\subseteq f(v,v)$ and $\beta(s)\subseteq f(w,w)$. Considering that $f$ preserves $S_{k-2}$ it follows that $|f(v,v)\cap f(w,w)|=k-2$. This implies that $\beta(r)\neq \beta(s)$, and therefore $|\beta(r)\cup \beta(s)|=k$.

	Now let $u,v\in {\mathbb{N}\choose k}$ be arbitrary, and let $r,s \in {\mathbb{N}\choose k-1}$ be distinct subsets of $u$. Then by definition we have $\beta(r)\cup \beta(s)\subseteq f(u,v),f(u,u)$. By Claim 3 we know that $|\beta(r)\cup \beta(s)|=k$ which means that in fact $f(u,v)=f(u,u)$.
	Let $\alpha \in \End(\bC)$ be the map $u\mapsto f(u,u)$.
	Then $f(x_1,x_2)=\alpha(x_1)$ holds for all $x_1,x_2 \in C$, which concludes the proof that $\Pol(\bC)$ is essentially unary.
\end{proof}

\begin{rem}
	The argument presented in our proof of Theorem~\ref{thm:J-coll} can also be generalised to polymorphisms of higher arity. This way one can avoid the reliance on Theorem~\ref{essential_binary} in the proof.
\end{rem}

\section{Conclusion and Open Problems}
 We verified the infinite-domain CSP dichotomy conjecture for all first-order reducts of the Johnson graph $\bJ(k)$, for all $k \in {\mathbb N}$. This result can be seen as a step of a larger program to verify the tractability conjecture for reducts of finitely bounded homogeneous structures along model-theoretic notions of tameness.
An important next step would be the verification of this conjecture for the class ${\mathcal K}_=$ of all structures $\bB$ with a first-order interpretation in $({\mathbb N};=)$ (see Figure~\ref{fig:classes}). This class includes $\bJ(k)$ and its first-order reducts, but also includes all finite structures and more generally cellular structures (which is the class of all structures with at most $cn^{dn}$ many orbits of $n$-tuples with pairwise distinct entries, for some constants $c,d$ with $d<1$~\cite{BodirskyBodor}),
and even more generally hereditarily cellular structures~\cite{Lachlan-Tree-Decomp,Braunfeld-Monadic-Stab,Bodor24}.
For these structures, the tractability conjecture has already been verified~\cite{BodorDiss}.
Interestingly, the class ${\mathcal K}_=$ is not closed under taking model-complete cores~\cite{bodirsky2025takingmodelcompletecores}.

A natural continuation of the present work would be the classification of the complexity of the CSP for all structures interpretable in $({\mathbb N};=)$ with a \emph{primitive} automorphism group, and then more generally with a \emph{transitive} automorphism group. Note that $\bJ(k)$ and all its first-order reducts are examples of such structures with a primitive automorphism group, but there are more.

\begin{exa}
Let $\bA$ be the structure with domain ${{\mathbb N}^{2}}$ and the relation $$D := \{((u_1,u_2),(v_1,v_2)) \mid u_1=v_1 \vee u_2 = v_2\}.$$
Clearly, this structure has a first-order interpretation in
$({\mathbb N};=)$.
Note that its automorphism group is a primitive action of the wreath product $\Sym({\mathbb N}) \Wr \Sym(\{1,2\})$.
\end{exa}

\noindent
Another interesting direction is to study
the action of other automorphism groups
$\Aut(\bA)$
on the set of $n$-element subsets of $A$,  for some $n \in {\mathbb N}$, and the respective first-order reducts. Note that any classification of such reducts contains the structures studied in the present article as a subcase, and we believe that our classification will be a useful point of departure in such an endeavour.

Quite remarkably, `almost all' first-order reducts of $\bJ(k)$ are NP-hard, in a sense that we now make precise.
We call a countable $\omega$-categorical structure
$\bB$ \emph{stubborn} if every model-complete core structure $\bC$ which is first-order  interdefinable with $\bB$ primitively positively constructs all finite structures.
Note that Theorem~\ref{thm:J-stbb} shows that
for any $k \geq 2$, the structure $\bJ(k)$ is stubborn. Other examples of stubborn structures are
the three non-trivial first-order reducts of $({\mathbb Q};<)$~\cite{tcsps-journal},
the countable universal local order  $S(2)$~\cite{BodirskyNesetrilJLC}, and
the Fra\"{i}ss\'e-limit of Boron trees (whose CSP is the phylogenetic reconstruction problem for the quartet relation;~\cite{Phylo-Complexity}).
Is it true that a countably infinite model-complete core of a first-order reduct of a stubborn structure is
again stubborn, or first-order bi-definable with $({\mathbb N};=)$ (which clearly is not stubborn)?

Note that in Section~\ref{sect:collapse} we also proved a stronger version of stubbornness. Following~\cite{marimon2024minimal}, we say that a permutation group $G$ is \emph{collapsing} if the only clone whose unary operations are exactly $\overline G$ is the one generated by $\overline G$. Theorem~\ref{thm:J-coll} shows that $\Aut(\bJ(k))$ is collapsing for every $k\geq 2$. Note that every structure 
with at least two elements and a  
collapsing automorphism group is stubborn (see the discussion at the beginning of Section~\ref{sect:collapse}). Many finite groups are known to be collapsing~\cite{palfy1982unary,HaddadRosenberg,kearnes2001collapsing}, but we only know about very few examples for infinite oligomorphic groups which are collapsing (see~\cite{marimon2024minimal}). In particular, none of the stubborn structures listed in the previous paragraph (apart from $\bJ(k)$) have collapsing automorphism groups.

\bibliographystyle{alphaurl}
\bibliography{LMCS-final}

\newcommand{\etalchar}[1]{$^{#1}$}
\begin{thebibliography}{BKO{\etalchar{+}}19}

\bibitem[BB21]{BodirskyBodor}
Manuel Bodirsky and Bertalan Bodor.
\newblock Permutation groups with small orbit growth.
\newblock {\em Journal of Group Theory}, 24(4):643--709, 2021.
\newblock \href {https://arxiv.org/abs/1810.05657} {\path{arXiv:1810.05657}},
  \href {https://doi.org/10.1515/jgth-2018-0220}
  {\path{doi:10.1515/jgth-2018-0220}}.

\bibitem[BB24]{BodirskyBodorUIPJournal}
Manuel Bodirsky and Bertalan Bodor.
\newblock A complexity dichotomy in spatial reasoning via {R}amsey theory.
\newblock {\em ACM Transactions on Computation Theory}, 16(2):10:1--10:39,
  2024.
\newblock \href {https://doi.org/10.1145/3649445} {\path{doi:10.1145/3649445}}.

\bibitem[BBM25]{bodirsky2025takingmodelcompletecores}
Manuel Bodirsky, Bertalan Bodor, and Paolo Marimon.
\newblock Taking model-complete cores, 2025.
\newblock \href {https://arxiv.org/abs/2512.21278} {\path{arXiv:2512.21278}}.

\bibitem[BG26]{BG-Sandwich}
Manuel Bodirsky and Santiago Guzm{\'{a}}n{-}Pro.
\newblock A {CSP} approach to graph sandwich problems.
\newblock In Kasper~Green Larsen and Barna Saha, editors, {\em Proceedings of
  the 2026 Annual {ACM-SIAM} Symposium on Discrete Algorithms, {SODA} 2026,
  Vancouver, BC, Canada, January 11-14, 2026}, pages 2403--2418. {SIAM}, 2026.
\newblock \href {https://doi.org/10.1137/1.9781611978971.85}
  {\path{doi:10.1137/1.9781611978971.85}}.

\bibitem[BHM12]{BodHilsMartin-Journal}
Manuel Bodirsky, Martin Hils, and Barnaby Martin.
\newblock On the scope of the universal-algebraic approach to constraint
  satisfaction.
\newblock {\em Logical Methods in Computer Science}, 8(3):1--30, 2012.
\newblock An extended abstract that announced some of the results appeared in
  the proceedings of Logic in Computer Science (LICS'10).
\newblock \href {https://doi.org/10.2168/LMCS-8(3:13)2012}
  {\path{doi:10.2168/LMCS-8(3:13)2012}}.

\bibitem[BJ10]{BodJunker}
Manuel Bodirsky and Markus Junker.
\newblock $\aleph_0$-categorical structures: endomorphisms and interpretations.
\newblock {\em Algebra Universalis}, 64(3-4):403--417, 2010.
\newblock \href {https://doi.org/10.1007/s00012-011-0110-y}
  {\path{doi:10.1007/s00012-011-0110-y}}.

\bibitem[BJ17]{Qualitative-Survey}
Manuel Bodirsky and Peter Jonsson.
\newblock A model-theoretic view on qualitative constraint reasoning.
\newblock {\em Journal of Artificial Intelligence Research}, 58:339--385, 2017.
\newblock \href {https://doi.org/10.1613/jair.5260}
  {\path{doi:10.1613/jair.5260}}.

\bibitem[BJP17]{Phylo-Complexity}
Manuel Bodirsky, Peter Jonsson, and Trung~Van Pham.
\newblock The complexity of phylogeny constraint satisfaction problems.
\newblock {\em ACM Transactions on Computational Logic}, 18(3):23:1--23:42,
  2017.
\newblock An extended abstract appeared in the conference STACS 2016.
\newblock \href {https://doi.org/10.1145/3105907} {\path{doi:10.1145/3105907}}.

\bibitem[BK08]{ecsps}
Manuel Bodirsky and Jan K\'ara.
\newblock The complexity of equality constraint languages.
\newblock {\em Theory of Computing Systems}, 43(2):136--158, 2008.
\newblock A conference version appeared in the proceedings of Computer Science
  Russia {(CSR'06)}.
\newblock \href {https://doi.org/10.1007/s00224-007-9083-9}
  {\path{doi:10.1007/s00224-007-9083-9}}.

\bibitem[BK10]{tcsps-journal}
Manuel Bodirsky and Jan K\'ara.
\newblock The complexity of temporal constraint satisfaction problems.
\newblock {\em Journal of the ACM}, 57(2):1--41, 2010.
\newblock An extended abstract appeared in the Proceedings of the Symposium on
  Theory of Computing (STOC).
\newblock \href {https://doi.org/10.1145/1667053.1667058}
  {\path{doi:10.1145/1667053.1667058}}.

\bibitem[BK12]{Cyclic}
Libor Barto and Marcin Kozik.
\newblock Absorbing subalgebras, cyclic terms, and the constraint satisfaction
  problem.
\newblock {\em Logical Methods in Computer Science}, 8(1):1--26, 2012.
\newblock \href {https://doi.org/10.2168/LMCS-8(1:7)2012}
  {\path{doi:10.2168/LMCS-8(1:7)2012}}.

\bibitem[BKO{\etalchar{+}}19]{BKOPP-equations}
Libor Barto, Michael Kompatscher, Miroslav Ol\v{s}\'{a}k, Trung~Van Pham, and
  Michael Pinsker.
\newblock Equations in oligomorphic clones and the constraint satisfaction
  problem for $\omega$-categorical structures.
\newblock {\em Journal of Mathematical Logic}, 19(2):\#1950010, 2019.
\newblock \href {https://doi.org/10.1142/S0219061319500107}
  {\path{doi:10.1142/S0219061319500107}}.

\bibitem[BKR21]{BKR}
Manuel Bodirsky, Simon Kn\"{a}uer, and Sebastian Rudolph.
\newblock Datalog-expressibility for monadic and guarded second-order logic.
\newblock In {\em 48th International Colloquium on Automata, Languages, and
  Programming, {ICALP} 2021, July 12-16, 2021, Glasgow, Scotland (Virtual
  Conference)}, pages 120:1--120:17, 2021.
\newblock \href {https://arxiv.org/abs/2010.05677} {\path{arXiv:2010.05677}},
  \href {https://doi.org/10.4230/LIPIcs.ICALP.2021.120}
  {\path{doi:10.4230/LIPIcs.ICALP.2021.120}}.

\bibitem[BKS20]{BodirskyKnaeuerStarke}
Manuel Bodirsky, Simon Kn{\"{a}}uer, and Florian Starke.
\newblock {ASNP:} {A} tame fragment of existential second-order logic.
\newblock In Marcella Anselmo, Gianluca~Della Vedova, Florin Manea, and Arno
  Pauly, editors, {\em Beyond the Horizon of Computability - 16th Conference on
  Computability in Europe, CiE 2020, Fisciano, Italy, June 29 - July 3, 2020,
  Proceedings}, volume 12098 of {\em Lecture Notes in Computer Science}, pages
  149--162. Springer, 2020.
\newblock \href {https://doi.org/10.1007/978-3-030-51466-2\_13}
  {\path{doi:10.1007/978-3-030-51466-2\_13}}.

\bibitem[BM18]{BodMot-Unary}
Manuel Bodirsky and Antoine Mottet.
\newblock A dichotomy for first-order reducts of unary structures.
\newblock {\em Logical Methods in Computer Science}, 14(2):1--31, 2018.
\newblock \href {https://doi.org/10.23638/LMCS-14(2:13)2018}
  {\path{doi:10.23638/LMCS-14(2:13)2018}}.

\bibitem[BN06]{BodirskyNesetrilJLC}
Manuel Bodirsky and Jaroslav Ne\v{s}et\v{r}il.
\newblock Constraint satisfaction with countable homogeneous templates.
\newblock {\em Journal of Logic and Computation}, 16(3):359--373, 2006.
\newblock \href {https://doi.org/10.1093/logcom/exi083}
  {\path{doi:10.1093/logcom/exi083}}.

\bibitem[Bod07]{Cores-journal}
Manuel Bodirsky.
\newblock Cores of countably categorical structures.
\newblock {\em Logical Methods in Computer Science}, 3(1):1--16, 2007.
\newblock \href {https://doi.org/10.2168/LMCS-3(1:2)2007}
  {\path{doi:10.2168/LMCS-3(1:2)2007}}.

\bibitem[Bod21]{Book}
Manuel Bodirsky.
\newblock {\em Complexity of Infinite-Domain Constraint Satisfaction},
  volume~52 of {\em Lecture Notes in Logic}.
\newblock Cambridge University Press, Cambridge, 2021.
\newblock \href {https://doi.org/10.1017/9781107337534}
  {\path{doi:10.1017/9781107337534}}.

\bibitem[Bod22]{BodorDiss}
Bertalan Bodor.
\newblock {\em {CSP} dichotomy for $\omega$-categorical monadically stable
  structures}.
\newblock PhD thesis, TU Dresden, 2022.
\newblock URL:
  \url{https://nbn-resolving.org/urn:nbn:de:bsz:14-qucosa2-774379}.

\bibitem[Bod24]{Bodor24}
Bertalan Bodor.
\newblock Classification of $\omega$-categorical monadically stable structures.
\newblock {\em Journal of Symbolic Logic}, 89(2):460--495, 2024.
\newblock \href {https://doi.org/10.1017/jsl.2023.66}
  {\path{doi:10.1017/jsl.2023.66}}.

\bibitem[BOP18]{wonderland}
Libor Barto, Jakub Opr\v{s}al, and Michael Pinsker.
\newblock The wonderland of reflections.
\newblock {\em Israel Journal of Mathematics}, 223(1):363--398, 2018.
\newblock \href {https://doi.org/10.1007/s11856-017-1621-9}
  {\path{doi:10.1007/s11856-017-1621-9}}.

\bibitem[BP11]{BP-reductsRamsey}
Manuel Bodirsky and Michael Pinsker.
\newblock Reducts of {R}amsey structures.
\newblock In Martin Grohe and Johann~A. Makowsky, editors, {\em Model Theoretic
  Methods in Finite Combinatorics}, volume 558 of {\em Contemporary
  Mathematics}, pages 489--519. American Mathematical Society, 2011.
\newblock \href {https://doi.org/10.1090/conm/558/11058}
  {\path{doi:10.1090/conm/558/11058}}.

\bibitem[BP20]{Topo}
Libor Barto and Michael Pinsker.
\newblock Topology is irrelevant (in a dichotomy conjecture for infinite domain
  constraint satisfaction problems).
\newblock {\em SIAM Journal on Computing}, 49(2):365--393, 2020.
\newblock \href {https://doi.org/10.1137/18M1216213}
  {\path{doi:10.1137/18M1216213}}.

\bibitem[BP21]{BP-canonical}
Manuel Bodirsky and Michael Pinsker.
\newblock Canonical functions: a proof via topological dynamics.
\newblock {\em Contributions to Discrete Mathematics}, 16(2):36--45, 2021.
\newblock \href {https://doi.org/10.55016/ojs/cdm.v16i2.71724}
  {\path{doi:10.55016/ojs/cdm.v16i2.71724}}.

\bibitem[BPP21]{BPP-projective-homomorphisms}
Manuel Bodirsky, Michael Pinsker, and Andr\'{a}s Pongr\'acz.
\newblock Projective clone homomorphisms.
\newblock {\em Journal of Symbolic Logic}, 86(1):148--161, 2021.
\newblock \href {https://doi.org/10.1017/jsl.2019.23}
  {\path{doi:10.1017/jsl.2019.23}}.

\bibitem[BPT13]{BPT-decidability-of-definability}
Manuel Bodirsky, Michael Pinsker, and Todor Tsankov.
\newblock Decidability of definability.
\newblock {\em Journal of Symbolic Logic}, 78(4):1036--1054, 2013.
\newblock A conference version appeared in the Proceedings of the Twenty-Sixth
  Annual IEEE Symposium on Logic in Computer Science (LICS 2011), pages
  321--328.
\newblock \href {https://doi.org/10.2178/jsl.7804020}
  {\path{doi:10.2178/jsl.7804020}}.

\bibitem[BR23]{RydvalDescr}
Manuel Bodirsky and Jakub Rydval.
\newblock On the descriptive complexity of temporal constraint satisfaction
  problems.
\newblock {\em Journal of the ACM}, 70(1):2:1--2:58, 2023.
\newblock A conference version of this article appeared under the title
  ``Temporal Constraint Satisfaction Problems in Fixed-Point Logic'' in
  LICS'20.
\newblock \href {https://doi.org/10.1145/3566051} {\path{doi:10.1145/3566051}}.

\bibitem[Bra22]{Braunfeld-Monadic-Stab}
Samuel Braunfeld.
\newblock Monadic stability and growth rates of $\omega$-categorical
  structures.
\newblock {\em Proceedings of the London Mathematical Society},
  124(3):373--386, 2022.
\newblock \href {https://arxiv.org/abs/1910.04380} {\path{arXiv:1910.04380}},
  \href {https://doi.org/10.1112/plms.12429} {\path{doi:10.1112/plms.12429}}.

\bibitem[Bul17]{BulatovFVConjecture}
Andrei~A. Bulatov.
\newblock A dichotomy theorem for nonuniform {CSP}s.
\newblock In {\em 58th {IEEE} Annual Symposium on Foundations of Computer
  Science, {FOCS} 2017, {B}erkeley, {CA}, {USA}, {O}ctober 15-17}, pages
  319--330, 2017.
\newblock \href {https://doi.org/10.1109/FOCS.2017.37}
  {\path{doi:10.1109/FOCS.2017.37}}.

\bibitem[Cam90]{Oligo}
Peter~J. Cameron.
\newblock {\em Oligomorphic Permutation Groups}, volume 152 of {\em London
  Mathematical Society Lecture Note Series}.
\newblock Cambridge University Press, Cambridge, 1990.
\newblock \href {https://doi.org/10.1017/CBO9780511549809}
  {\path{doi:10.1017/CBO9780511549809}}.

\bibitem[CHL85]{CherlinHarringtonLachlan}
Gregory Cherlin, Leo Harrington, and Alistair~H. Lachlan.
\newblock $\aleph_0$-categorical, $\aleph_0$-stable structures.
\newblock {\em Annals of Pure and Applied Logic}, 28(2):103--135, 1985.
\newblock \href {https://doi.org/10.1016/0168-0072(85)90023-5}
  {\path{doi:10.1016/0168-0072(85)90023-5}}.

\bibitem[CL86]{CherlinLachlan}
Gregory Cherlin and Alistair~H. Lachlan.
\newblock Stable finitely homogeneous structures.
\newblock {\em Transactions of the American Mathematical Society},
  296(2):815--850, 1986.
\newblock \href {https://doi.org/10.1090/S0002-9947-1986-0846608-8}
  {\path{doi:10.1090/S0002-9947-1986-0846608-8}}.

\bibitem[Dez74]{Deza}
M.~Deza.
\newblock Solution d'un probl{\`e}me de {E}rd{\"o}s-{L}ov{\'a}sz.
\newblock {\em Journal of Combinatorial Theory, Series B}, 16(2):166--167,
  1974.

\bibitem[DM96]{DixonMortimer}
John~D. Dixon and Brian Mortimer.
\newblock {\em Permutation Groups}, volume 163 of {\em Graduate Texts in
  Mathematics}.
\newblock Springer, New York, 1996.
\newblock \href {https://doi.org/10.1007/978-1-4612-0731-3}
  {\path{doi:10.1007/978-1-4612-0731-3}}.

\bibitem[DNT86]{DixonNeumannThomas}
John~D. Dixon, Peter~M. Neumann, and Simon Thomas.
\newblock Subgroups of small index in infinite symmetric groups.
\newblock {\em Bulletin of the London Mathematical Society}, 18(6):580--586,
  1986.
\newblock \href {https://doi.org/10.1112/blms/18.6.580}
  {\path{doi:10.1112/blms/18.6.580}}.

\bibitem[FV98]{FederVardi}
Tom\'as Feder and Moshe~Y. Vardi.
\newblock The computational structure of monotone monadic {SNP} and constraint
  satisfaction: {a} study through {D}atalog and group theory.
\newblock {\em {SIAM} Journal on Computing}, 28(1):57--104, 1998.
\newblock \href {https://doi.org/10.1137/S0097539794266766}
  {\path{doi:10.1137/S0097539794266766}}.

\bibitem[HHLS93]{HodgesHodkinsonLascarShelah}
Wilfrid Hodges, Ian Hodkinson, Daniel Lascar, and Saharon Shelah.
\newblock The small index property for $\omega$-stable $\omega$-categorical
  structures and for the random graph.
\newblock {\em Journal of the London Mathematical Society}, s2-48(2):204--218,
  1993.
\newblock \href {https://doi.org/10.1112/jlms/s2-48.2.204}
  {\path{doi:10.1112/jlms/s2-48.2.204}}.

\bibitem[Hod97]{Hodges}
Wilfrid Hodges.
\newblock {\em A Shorter Model Theory}.
\newblock Cambridge University Press, Cambridge, 1997.

\bibitem[HR94]{HaddadRosenberg}
Lucien Haddad and Ivo~G. Rosenberg.
\newblock Finite clones containing all permutations.
\newblock {\em Canadian Journal of Mathematics}, 46(5):951--970, 1994.
\newblock \href {https://doi.org/10.4153/CJM-1994-054-1}
  {\path{doi:10.4153/CJM-1994-054-1}}.

\bibitem[Hru89]{HrushovskiTotallyCategorical}
Ehud Hrushovski.
\newblock Totally categorical structures.
\newblock {\em Transactions of the American Mathematical Society},
  313(1):131--159, 1989.
\newblock \href {https://doi.org/10.1090/S0002-9947-1989-0943605-1}
  {\path{doi:10.1090/S0002-9947-1989-0943605-1}}.

\bibitem[KPT05]{Topo-Dynamics}
Alexander Kechris, Vladimir Pestov, and Stevo Todor\v{c}evi\'c.
\newblock Fra\"{i}ss\'e limits, {R}amsey theory, and topological dynamics of
  automorphism groups.
\newblock {\em Geometric and Functional Analysis}, 15(1):106--189, 2005.
\newblock \href {https://doi.org/10.1007/s00039-005-0503-1}
  {\path{doi:10.1007/s00039-005-0503-1}}.

\bibitem[KS01]{kearnes2001collapsing}
Keith~A. Kearnes and {\'A}gnes Szendrei.
\newblock Collapsing permutation groups.
\newblock {\em Algebra Universalis}, 45(1):35--51, 2001.
\newblock \href {https://doi.org/10.1007/s000120050200}
  {\path{doi:10.1007/s000120050200}}.

\bibitem[Lac87]{LachlanIndiscernible}
Alistair~H. Lachlan.
\newblock Structures coordinatized by indiscernible sets.
\newblock {\em Annals of Pure and Applied Logic}, 34(3):245--273, 1987.
\newblock \href {https://doi.org/10.1016/0168-0072(87)90003-0}
  {\path{doi:10.1016/0168-0072(87)90003-0}}.

\bibitem[Lac92]{Lachlan-Tree-Decomp}
Alistair~H. Lachlan.
\newblock $\aleph_0$-categorical tree-decomposable structures.
\newblock {\em Journal of Symbolic Logic}, 57(2):501--514, 1992.
\newblock \href {https://doi.org/10.2307/2275284} {\path{doi:10.2307/2275284}}.

\bibitem[MP21]{MottetPinskerCores}
Antoine Mottet and Michael Pinsker.
\newblock Cores over {R}amsey structures.
\newblock {\em Journal of Symbolic Logic}, 86(1):352--361, 2021.
\newblock \href {https://doi.org/10.1017/jsl.2021.6}
  {\path{doi:10.1017/jsl.2021.6}}.

\bibitem[MP22]{MottetPinskerSmoothConf}
Antoine Mottet and Michael Pinsker.
\newblock Smooth approximations and {CSP}s over finitely bounded homogeneous
  structures.
\newblock In {\em Proceedings of the 37th Annual {ACM/IEEE} Symposium on Logic
  in Computer Science ({LICS} 2022)}, 2022.
\newblock \href {https://arxiv.org/abs/2011.03978} {\path{arXiv:2011.03978}},
  \href {https://doi.org/10.1145/3531130.3533353}
  {\path{doi:10.1145/3531130.3533353}}.

\bibitem[MP24]{marimon2024minimal}
Paolo Marimon and Michael Pinsker.
\newblock Minimal operations over permutation groups, 2024.
\newblock \href {https://arxiv.org/abs/2410.22060} {\path{arXiv:2410.22060}}.

\bibitem[PAS82]{palfy1982unary}
P{\'e}ter~P{\'a}l P{\'a}lfy and \'{A}gnes Szendrei.
\newblock Unary polynomials in algebras, {II}.
\newblock {\em Contributions to General Algebra}, 2:273--290, 1982.

\bibitem[Ram30]{Ram:On-a-problem}
Frank~Plumpton Ramsey.
\newblock On a problem of formal logic.
\newblock {\em Proceedings of the London Mathematical Society},
  s2-30(1):264--286, 1930.
\newblock \href {https://doi.org/10.1112/plms/s2-30.1.264}
  {\path{doi:10.1112/plms/s2-30.1.264}}.

\bibitem[Tho96]{Thomas96}
Simon Thomas.
\newblock Reducts of random hypergraphs.
\newblock {\em Annals of Pure and Applied Logic}, 80(2):165--193, 1996.
\newblock \href {https://doi.org/10.1016/0168-0072(95)00061-5}
  {\path{doi:10.1016/0168-0072(95)00061-5}}.

\bibitem[Zhu17]{ZhukFVConjecture}
Dmitriy~N. Zhuk.
\newblock A proof of {CSP} dichotomy conjecture.
\newblock In {\em 58th {IEEE} Annual Symposium on Foundations of Computer
  Science, {FOCS} 2017, {B}erkeley, {CA}, {USA}, {O}ctober 15-17}, pages
  331--342, 2017.
\newblock \href {https://arxiv.org/abs/1704.01914} {\path{arXiv:1704.01914}},
  \href {https://doi.org/10.1109/FOCS.2017.38}
  {\path{doi:10.1109/FOCS.2017.38}}.

\bibitem[Zhu20]{Zhuk20}
Dmitriy Zhuk.
\newblock A proof of the {CSP} dichotomy conjecture.
\newblock {\em Journal of the {ACM}}, 67(5):30:1--30:78, 2020.
\newblock \href {https://doi.org/10.1145/3402029} {\path{doi:10.1145/3402029}}.

\end{thebibliography}

\end{document}